\documentclass[12pt]{extarticle}
\usepackage{amsmath, amsthm, amssymb, color}
\usepackage[colorlinks=true,linkcolor=blue,urlcolor=blue]{hyperref}
\usepackage{graphicx}
\usepackage{caption}
\usepackage{mathtools}
\usepackage{enumerate}
\usepackage{verbatim}
\usepackage{tikz,tikz-cd,tikz-3dplot}
\usepackage{amssymb}
\usetikzlibrary{matrix}
\usetikzlibrary{arrows}
\usepackage{algorithm}
\usepackage[noend]{algpseudocode}
\usepackage{caption}
\usepackage[normalem]{ulem}
\usepackage{subcaption}
\usepackage{multicol}
\oddsidemargin \evensidemargin
\usepackage{makecell}
\usepackage{array}
\usepackage{enumitem}

\newtheorem{theorem}{Theorem}[section]

\newtheorem{proposition}[theorem]{Proposition}
\newtheorem{lemma}[theorem]{Lemma}
\newtheorem{corollary}[theorem]{Corollary}

\theoremstyle{definition}

\newtheorem{remark}[theorem]{Remark}

\newtheorem{examplex}[theorem]{Example}
\newenvironment{example}
{\pushQED{\qed}\examplex}
{\popQED\endexamplex}

\newcommand{\PP}{\mathbb{P}}

\newcommand{\CC}{\mathbb{C}}
\newcommand{\ZZ}{\mathbb{Z}}

\title{\bf The Chow-Lam Form}
\author{Elizabeth Pratt\footnote{Supported by an NSF Graduate Research Fellowship } \ and Bernd Sturmfels}
\date{}
\begin{document}
\maketitle

\begin{abstract}
  \noindent
 The classical Chow form encodes any projective variety by one equation.
 We here introduce the Chow-Lam form for subvarieties of a Grassmannian.
By evaluating the Chow-Lam form at twistor coordinates, we 
obtain universal projection formulas. These were pioneered
by Thomas Lam for positroid varieties in the study of amplituhedra,
and we develop his approach further. Universal formulas for branch loci are obtained from
Hurwitz-Lam forms. Our focus is on computations and applications in geometry.
\end{abstract}

\section{Introduction}

The Grassmannian ${\rm Gr}(k,n)$ is
a projective variety of dimension $k(n-k)$
that is embedded in $\PP^{\binom{n}{k}-1}$.
Its points  are linear subspaces of dimension 
$k$ in $\CC^n$. Such a subspace is usually given as the kernel or 
 row space of a matrix, whose maximal minors furnish
primal or dual  Pl\"ucker coordinates. Either of these specifies the
Pl\"ucker embedding of ${\rm Gr}(k,n)$ into $\PP^{\binom{n}{k}-1}$.
In the special cases $k=1$ and $k=n-1$, the Grassmannian 
is the projective space $\PP^{n-1}$.

Suppose we are given a subvariety $\mathcal{V}$ of ${\rm Gr}(k,n)$,
where ${\rm dim}(\mathcal{V}) = k(r-k)-1$ for some $r \in \{k+1,\ldots,n\}$.
Our aim is to characterize $\mathcal{V}$ by a single equation.
This will generalize the classical Chow form \cite{Cay, CvW, DS, GKZ}
from $\PP^{n-1}$ to $k \in \{2,3,\ldots,n-2\}$.
Let $\mathcal{CL}_\mathcal{V}$ denote the set of linear spaces
$P \in {\rm Gr}(k+n-r,n)$ which 
contain a subspace $Q$ belonging to $\mathcal{V}$.
The codimension of $\mathcal{CL}_\mathcal{V}$ 
in ${\rm Gr}(k+n-r,n)$ is expected to be one, by equation (\ref{eq:dof}).
If it is one, then
$\mathcal{CL}_\mathcal{V}$    is defined by a homogeneous
polynomial  in Pl\"ucker coordinates, which is unique modulo Pl\"ucker relations.
 This polynomial is denoted
${\rm CL}_\mathcal{V}$ and called the
 {\em Chow-Lam form} of $\mathcal{V}$. 
 For $k \in \{1,n-1\}$, this specializes to
the Chow form of a projective variety $\mathcal{V} \subset \PP^{n-1}$.

\begin{example}[Ruled surfaces in 3-space] \label{ex:ruledsurface}
Fix $k=2,n=4,r=3$, so that $\mathcal{V}$ is
a curve is the Grassmannian ${\rm Gr}(2,4)$.
This curve parametrizes lines $Q$ in $\PP^3$ whose union is
a ruled surface $\mathcal{S}$. Then $\mathcal{CL}_\mathcal{V}$
is a surface in ${\rm Gr}(3,4) = (\PP^3)^\vee$. It
parametrizes planes $P$ in $\PP^3$
which contain a line $Q$ from the curve $\mathcal{V}$.
Hence, $\mathcal{CL}_\mathcal{V}$ is the
surface projectively dual to the ruled surface~$\mathcal{S}$.
For a concrete example, consider the curve of lines that intersect three given lines.
Each line is given by its Chow form, whose coefficients are the primal Pl\"ucker coordinates:
$$ \mathcal{V} \,=\, V \bigl(\,
  q_{12}+q_{14}-2 q_{23}+2 q_{34}\,,\,\,
  q_{13}+2 q_{14}+q_{23}+2 q_{24}\,,\,\,
  5q_{12} + 2q_{14} - 25q_{23} + 10q_{34}
  \, \bigr) \,\,\subset \,\,{\rm Gr}(2,4) .
   $$
   The varying line $Q \subset \PP^3$ is represented by its dual Pl\"ucker coordinates
on ${\rm Gr}(2,4) = V(q_{12} q_{34} - q_{13} q_{24} + q_{14} q_{23})$.
To recover the ruled surface, we augment the three linear forms above by
the four equations $\,q_{ij} x_k - q_{ik} x_j + q_{ij} x_k = 0$, which cut out the line  $Q $.
Next we saturate and then we eliminate the $q$-variables.
This yields the corresponding ruled surface
$$ \mathcal{S} \,=\, V \bigl(\,
x_1^2-9 x_1 x_2-10 x_2^2+7 x_1 x_3+10 x_3^2-14 x_2 x_4+18 x_3 x_4-4 x_4^2 \,\bigr) \,\,\,
\subset \,\,\, \PP^3. $$
Note that $\mathcal{V}$ is one of the two rulings of the quadric $\mathcal{S}$.
A plane $P$ with primal coordinates $p_1,p_2,p_3,p_4$ contains
the line $Q$ if and only if $\,\sum_{i=1}^4 q_{ij} p_i = 0$ for all $j$.
Elimination now yields
\begin{equation}
\label{eq:chowlamex} {\rm CL}_\mathcal{V} \, = \,
20 p_1^2-18 p_1 p_2-2 p_2^2-14 p_1 p_3+2 p_3^2+7 p_2 p_4+9 p_3 p_4-5 p_4^2 .
\end{equation}
This defines
 the Chow-Lam locus $ \mathcal{CL}_\mathcal{V}$  in
${\rm Gr}(3,4) = (\PP^3)^\vee$. This  quadric is dual
to $\mathcal{S} \subset \PP^3$.

The Chow-Lam form serves as a universal equation for  projections of $\mathcal{V}$.
Fix a  $3 \times 4$ matrix $Z = (z_{ij}) $. This defines a linear projection
${\rm Gr}(2,4) \dashrightarrow {\rm Gr}(2,3) =  (\PP^2)^\vee$.
We are interested in the image of our curve $\mathcal{V}$
in that dual projective plane, with  coordinates $y = (y_1,y_2,y_3)$.
The entries of the vector-matrix product $y Z$ are referred to as {\em twistor coordinates}.
Explicitly,
\begin{equation}
\label{eq:twistorex}
p_i \,=\,  z_{1i} y_1 + z_{2i} y_2 + z_{3i} y_3,
\qquad {\rm for} \quad i=1,2,3,4.
\end{equation}
Substituting (\ref{eq:twistorex}) into
 (\ref{eq:chowlamex}), we obtain a quadric in $y$ whose
coefficients depend on $Z$. This is the equation of the image curve.
The same recipe works for any ruled surface $\mathcal{S} \subset \PP^3$.
\end{example}

The name ``Chow-Lam form'' recognizes work of Thomas Lam \cite{LamCurrent, LamStanley}
at the interface of combinatorics and particle physics.
Lam focuses on the case when $\mathcal{V}$ is a positroid variety, defined by
special collections of Pl\"ucker coordinates on ${\rm Gr}(k,n)$.
Lam refers to $\mathcal{CL}_\mathcal{V}$ as
the {\em universal amplituhedron variety} \cite[Section 18.1]{LamCurrent},
and he discusses universal projections via
twistor coordinates. We now review Lam's
{\em degree-three example} from \cite[Section~19.4]{LamCurrent}.

\begin{example}[Chow-Lam form of a positroid] \label{ex:positroid92}
Fix $k{=}2, n {=} 9, r {=} 7$ and the positroid~variety
$$ \mathcal{V} \,= \, V( q_{12}, q_{13}, q_{23}, q_{45}, q_{67}, q_{89} ) \,\, \subset \,\, {\rm Gr}(2,9). $$
This is the positroid $\beta = (3,2,2,2)$ in the notation of Section \ref{sec4}. 
Then $\mathcal{CL}_\mathcal{V}$ is a hypersurface in ${\rm Gr}(4,9)$. It consists of all
$4 \times 9$-matrices $X $ whose columns, viewed as points in $\PP^3$, satisfy: 
some line in the plane $123$ meets the lines
$45$, $67$ and $89$. This happens if and only if the three lines intersect the
plane in collinear points. We can write these intersection points explicitly in Pl\"ucker coordinates. For example, 
intersection of $123$ and $45$ is the point $q_{2345} x_1 -q_{1345} x_2 + q_{1245} x_3,$ where $x_i$ is the $i$th column of 
$X$. Therefore, the Chow-Lam form~is
$$ \! {\rm CL}_\mathcal{V}  = 
{\rm det}\! \begin{small} \begin{bmatrix}
       q_{2345} \! &\!\! -q_{1345}\! &\! q_{1245} \\
       q_{2367} \! &\!\! -q_{1367} \!&\!  q_{1267} \\
       q_{2389} \!& \!\! -q_{1389} \!&\! q_{1289}
\end{bmatrix} \end{small} \! = 
q_{1234} q_{1237}q_{5689}+q_{1234} q_{1236} q_{5789} 
-q_{1235} q_{1236} q_{4789}-q_{1235} q_{1237} q_{4689}.
$$
Projections of $\mathcal{V}$ into ${\rm Gr}(2,7)$ satisfy equations in
twistor coordinates, by Corollary \ref{cor:projhyper}.
\end{example}

The goal of this article is to develop the theory and practise
 of Chow-Lam forms, using that of Chow forms as a guide. In this, we transition
from subvarieties of projective space to subvarieties
of Grassmannians. Our motivation for this
comes from  algebraic geometry, as in
Examples \ref{ex:ruledsurface} and \ref{ex:fano}, and from
combinatorics and physics, as in
Examples~\ref{ex:positroid92} and \ref{ex:dmitrii}.

We start in Section \ref{sec:two} with a review
of coordinate systems on Grassmannians
and basics on Chow forms, such as the Intersection Formula and
the Projection Formula.
In Section \ref{sec:three} we develop
the corresponding theory of  Chow-Lam forms. Building on
the work of Lam \cite{LamStanley}, we give
a criterion for $\mathcal{CL}_\mathcal{V}$ to have codimension one,~and 
we present a formula for the degree of ${\rm CL}_\mathcal{V}$.
In Section \ref{sec4} we turn to varieties given by
matroids and positroids.
We compute their Chow-Lam degrees and Chow-Lam forms
in some interesting cases, such as Theorem \ref{thm:catalan}.
In Section \ref{sec5} we introduce the Hurwitz-Lam form,
which governs non-transversal intersections of complementary dimension.
This generalizes \cite{Hur} and
can be used to compute branch loci.

\section{Coordinates and Chow Forms}\label{sec:two}

This paper develops tools for computing with
subvarieties in Grassmannians. To this end, it is important for us
to be very precise about the coordinates to be used.
We distinguish four different coordinate systems to represent
 a linear subspace $L \subset \CC^n$, corresponding
to a point in ${\rm Gr}(k,n)$. What follows
is consistent with the conventions adopted in \cite{DS, Cortona, Hur}.

If $L$ is given to us as the kernel of an $(n-k) \times n$-matrix then
the entries of that matrix are called the {\em primal Stiefel coordinates}
and its maximal minors are the
{\em primal Pl\"ucker coordinates},  denoted $p_{i_1 i_2 \cdots i_{n-k}}$.
If $L$ is given to us as the row space of a $k \times n$-matrix then
the entries of that matrix are called the {\em dual Stiefel coordinates}
and its maximal minors are the
{\em dual Pl\"ucker coordinates},  denoted $q_{j_1 j_2 \cdots j_k}$.
After complementing indices,
 primal and dual Pl\"ucker coordinates agree up to multiplication by $(-1)^{j_1+\cdots+j_k}$, indicating
the sign of the permutation $(j_1, ..., j_k)$ of $(1, ..., k)$.
For example, the ten Pl\"ucker
coordinates on ${\rm Gr}(3,5)$~are
\setcounter{MaxMatrixCols}{20}
\begin{equation}
\label{eq:primaldual35}
 \begin{matrix}
{\rm Primal} & & p_{12} & p_{13} & p_{14} & p_{15} & p_{23} & p_{24} & p_{25} & p_{34} &
p_{35} & p_{45}, \\ {\rm Dual} & &
q_{345} & -q_{245} & q_{235} & -q_{234} & q_{145} & -q_{135} & q_{134} & q_{125} & -q_{124} & q_{123}.
\end{matrix}
\end{equation}
In geometric applications, $L$  represents a projective subspace
of dimension $k-1$ in $\PP^{n-1}$. In the primal perspective,
$L$ is given as the intersection of hyperplanes.
In the  dual perspective, $L$ is the span of points.
Which of these is preferable depends on whether $k$ or
$n-k$ is smaller.

\begin{remark}
Pl\"ucker coordinates are always alternating with respect to permuting indices.
For instance, in (\ref{eq:primaldual35}) we have
$\,p_{13} = -p_{31} \,=\, -q_{245} = q_{254} = q_{425} = - q_{452} = -q_{524} = q_{542}$.
\end{remark}

Let $\mathcal{V}$ be an irreducible variety of dimension $d$ in $\PP^{n-1}$.
We now define the Chow form of $\mathcal{V}$.
Let $\mathcal{C}_\mathcal{V}$ be the subvariety of ${\rm Gr}(n-d-1,n)$ whose points are
subspaces $L$ such that $L \cap \mathcal{V} \not= \emptyset$
in $\PP^{n-1}$. Then $\mathcal{C}_\mathcal{V}$ has codimension one.
Since each Grassmannian has Picard group $\ZZ$,
the hypersurface $\mathcal{C}_\mathcal{V}$
is the zero set of a single polynomial in Pl\"ucker coordinates.
This polynomial is denoted
by ${\rm C}_\mathcal{V}$ and called the {\em Chow form} of $\mathcal{V}$.
It is unique up to the Pl\"ucker relations. 
We will show in Corollary \ref{cor:chowdegree} that the degree of ${\rm C}_\mathcal{V}$ in Pl\"ucker coordinates equals the
degree of $\mathcal{V}$ in~$\PP^{n-1}$.

The Chow form ${\rm C}_\mathcal{V}$ can be written in either primal Pl\"ucker coordinates,
 dual Pl\"ucker coordinates, primal Stiefel coordinates, or dual Stiefel coordinates.
All four variants are useful, depending on the context.
We illustrate this for the rational normal curve in $\PP^4$.

\begin{example}[$d=1,n=5$] \label{ex:quarticcurve}
Let $\mathcal{V}$ be the curve $(1:t:t^2:t^3:t^4)$ in $\PP^4$.
As in \cite[Section 1.2]{DS},
its Chow form  in primal Pl\"ucker coordinates  is the determinant of the  {\em B\'ezout matrix}:
\begin{equation}
\label{eq:bezout}
{\rm C}_\mathcal{V} \,\, = \,\, {\rm det}
\begin{bmatrix}
\,p_{12} & p_{13} & p_{14} & p_{15} \,\\
\,p_{13} & p_{14} + p_{23} & p_{15} + p_{24} & p_{25}\,  \\
\,p_{14} & p_{15} + p_{24} & p_{25} + p_{34} & p_{35}\,  \\
\,p_{15} & p_{25} & p_{35} & p_{45} \,
\end{bmatrix}.
\end{equation}
Passing to primal Stiefel coordinates $p_{ij} = a_i b_j - a_j b_i$, we obtain the {\em Sylvester resultant}
$$ {\rm C}_\mathcal{V} \,\, = \,\,
{\rm Res}_t \bigl( a_1 + a_2 t + a_3 t^2 + a_4 t^3 + a_5 t^4, 
\, b_1 + b_2 t + b_3 t^2 + b_4 t^3 + b_5 t^4\bigr). $$
For the formula in dual Pl\"ucker coordinates, replace
each $p_{ij}$ with the $\pm q_{klm}$ below it in (\ref{eq:primaldual35}).
Replacing the $q_{klm}$ with the $3 \times 3$ minors of
a $3 \times 5$ matrix, we obtain the formula for ${\rm C}_\mathcal{V}$
in dual Stiefel coordinates. This is a polynomial of degree $12$ in $15$ unknowns. It
characterizes triples of  binary quartics whose
linear span contains the fourth power
of some linear form.
\end{example}

An analogue to the B\'ezout formula exists for arbitrary curves in $\PP^{n-1}$.
This is explained in \cite[Section 4]{ES}. In their article \cite{ES},
Eisenbud and Schreyer  present a
method for computing determinantal formulas for Chow forms. 
This rests on syzygies for Ulrich sheaves on $\mathcal{V}$.
We are optimistic that this generalizes to Chow-Lam forms in the setting of Grassmannians.

But first we review basics on Chow forms, following the exposition in \cite{DS}.
We present the formulas for  intersections and projections of
projective varieties in terms of their Chow forms.  
We begin with intersections. The following result is found in \cite[Proposition 2.1]{DS}.

\begin{proposition}[Intersection Formula] \label{prop:chowintersec}
Suppose $dim(\mathcal{V}) = d$, and
let $L$ and $M$ be linear subspaces of $\PP^{n-1}$
such that ${\rm codim}(L \cap M) =
{\rm codim}(L) + {\rm codim}(M) = d+1$. Then
\begin{equation}
\label{eq:VVLM}
{\rm C}_{\mathcal{V} \cap L} (M) \, = \, {\rm C}_{\mathcal{V}}(L \cap M). 
\end{equation}
\end{proposition}

To use this formula in practice, we need to express the Pl\"ucker coordinates
of $L \cap M$ via those of $L$ and $M$.
Let $\ell$ and $m$ be the dual Pl\"ucker coordinates of $L$ and $M$
respectively. Then the dual Pl\"ucker coordinates of $L \cap M$ are given by
the exterior product of $\ell$ and $m$:
\begin{equation}
\label{ex:exterior}
 p \,\,=\,\, \ell \,\wedge\, m.
 \end{equation}
For example, let $n=6$ and $d=2$ and suppose
that ${\rm codim}(L) = 1$ and ${\rm codim}(M) = 2$.
Then $\ell = (\ell_1,\ell_2,\ldots,\ell_6)$,
$m = (m_{12}, m_{13}, \ldots, m_{56})$, and
the $20$ coordinates of (\ref{ex:exterior}) are  as follows:
\begin{equation}
\label{eq:plm} p_{ijk} \,\,= \,\,\ell_i m_{jk} - \ell_j m_{ik} + \ell_k m_{ij} \qquad
{\rm for}\quad 1 \leq i < j < k \leq 6. 
\end{equation}
We illustrate this formula in a concrete scenario of interest in elimination theory \cite{Cortona}.

\begin{example}[Veronese surface] \label{ex:veronese1}
Fix the surface $\mathcal{V} = \{(1:x:y:x^2:xy:y^2)\}$ in~$\PP^5$.
Its Chow form ${\rm C}_\mathcal{V}$ is the resultant of three ternary quadrics. Explicitly,
${\rm C}_\mathcal{V}$ is a polynomial of degree $4$ in 
primal Pl\"ucker coordinates $p_{ijk}$ on ${\rm Gr}(3,6)$. 
See \cite[Section~2.2]{Cortona} for the formula.
The expansion of ${\rm C}_\mathcal{V}$ into primal Stiefel coordinates has
$21894$ terms; see \cite[eqn (4.5)]{CBMS}.

Let $L$ be a  hyperplane in $\PP^5$, with coordinates $\ell_i$. The curve
$\mathcal{V} \cap L$ is the Veronese embedding of the conic
$V( \ell_1 + \ell_2 x + \ell_3 y + \ell_4 x^2 + \ell_5 xy + \ell_6 y^2) \subset \PP^2$.
By substituting  (\ref{eq:plm}) into ${\rm C}_\mathcal{V}$, we obtain the Chow form
${\rm C}_{\mathcal{V} \cap L}$ of
this curve, written in primal Pl\"ucker coordinates $m_{ij}$.

Next let  ${\rm codim}(L) = 2$, with Pl\"ucker coordinates $\ell_{ij}$.
Then $\mathcal{V} \cap L$ consists of four points in $\PP^5$.
The Chow form ${\rm C}_{\mathcal{V} \cap L}(m_1,\ldots,m_6)$
is a quartic that factors into four linear forms.
\end{example}

\begin{corollary} \label{cor:chowdegree}
The degree of the Chow form ${\rm C}_\mathcal{V}$ equals the degree of the variety $\mathcal{V}$
in $\PP^{n-1}$.
\end{corollary}

\begin{proof} 
Let  ${\rm codim}(L) = d = {\rm dim}(\mathcal{V})$ in  Proposition \ref{prop:chowintersec}.
By (\ref{eq:VVLM}), the Chow forms ${\rm C}_\mathcal{V}$
and ${\rm C}_{\mathcal{V} \cap L} $ have the same
degree in their Pl\"ucker coordinates.
We claim that this degree is  $\delta = {\rm degree}(\mathcal{V})$.
This holds because
$\mathcal{V} \cap L$ is a finite set
$\,\bigl\{u^{(i)} = (u^{(i)}_1:\cdots:u^{(i)}_n)\,:\,i=1,2,\ldots,\delta \bigr\}$.
Its Chow form equals
${\rm C}_{\mathcal{V} \cap L}(m) =  \prod_{i=1}^\delta (u^{(i)}_1 m_1 + \cdots + u^{(i)}_n m_n)$.
This has degree $\delta$ in $m$. 
\end{proof}

We now turn to projections $\PP^{n-1} \dashrightarrow \PP^{r-1}$.
These are given by $r \times n$ matrices $Z = (z_{ij})$.
Let $\mathcal{V} \subset \PP^{n-1}$ be a variety of dimension
$d \leq r-2$.
We write $Z(\mathcal{V})$ for the closure of the image of $\mathcal{V}$ in $\PP^{r-1}$.
For general $Z$, the variety $Z(\mathcal{V})$ 
has  dimension $d$ and degree $\delta = {\rm degree}(\mathcal{V})$ in $\PP^{r-1}$.
The hypersurface $\mathcal{C}_{Z(\mathcal{V})}$ lives in ${\rm Gr}(r-d-1,r)$. We shall
write its defining polynomial  ${\rm C}_{Z (\mathcal{V})}$
in terms of dual Stiefel coordinates.
We represent an element of ${\rm Gr}(r-d-1,r)$ as the
column span of an $r \times (r-d-1)$ matrix
$Y$ with unknown entries. The concatenation
$\,[\, Z \,| \, Y \,] \,$ is a matrix with $r$ rows and $n+r-d-1$ columns.
For any sequence $1 \leq i_1 < i_2 < \cdots < i_{d+1} \leq n$, we introduce a
 {\em twistor coordinate}
$\,[\, Z \,| \, Y \,]_{i_1 i_2 \cdots i_{d+1}}$. This is the
 $r \times r $ subdeterminant of $\,[\, Z \,| \, Y \,] \,$
given by the $r-d-1$ columns of $Y$ and
the $d+1$ columns of $Z$ indexed by $i_1, i_2,\ldots,i_{d+1}$.

The twistor coordinate $[\, Z \,| \, Y \,]_{i_1 i_2 \cdots i_{d+1}}$
is a linear form in the maximal minors of the matrix $Y$.
These are dual Pl\"ucker coordinates on ${\rm Gr}(r-d-1,r)$.
They are preferred when $r-d$ is small. However,
if $d$ is small then it is better to use primal Pl\"ucker coordinates.
We find these by multiplying $Z$ on the left with a matrix
of primal Stiefel coordinates, as in~(\ref{eq:twistorex}).

\begin{proposition}[Projection Formula] \label{prop:chowprojection}
The Chow form ${\rm C}_{Z (\mathcal{V})}$ in dual Stiefel coordinates equals the
Chow form ${\rm C}_{\mathcal{V}}$ with
primal Pl\"ucker coordinates replaced with twistor coordinates:
\begin{equation}
\label{eq:twistor1}
  \qquad p_{i_1 i_2 \cdots i_{d+1}} \,\, = \,\, \,[\, Z \,| \, Y \,]_{i_1 i_2 \cdots i_{d+1}}\quad
{\rm for} \,\, 1 \leq i_1 < i_2 < \cdots < i_{d+1} \leq n. 
\end{equation}
\end{proposition}

\begin{proof}
This is a reinterpretation of the formula in \cite[Section 2.2]{DS}
which was derived for the projection from a point.
In more geometric terms, our formula can be written as follows:
$$ {\rm C}_{Z (\mathcal{V})}(Y) \, = \, {\rm C}_{\mathcal{V}} (Z^{-1}(Y)) .$$
Here $Y$ is a  subspace of dimension $r-d-2$ in $\PP^{r-1}$.
Its preimage $Z^{-1}(Y)$ is a  subspace of dimension $n-d-2$ in $\PP^{n-1}$.
We have $\,Y \cap Z(\mathcal{V}) \not= \emptyset\,$ if and only if
$\,Z^{-1}(Y) \cap \mathcal{V} \not= \emptyset$.
\end{proof}

The case of most interest in the projection formula is $r = d+2$, when
 $Z(\mathcal{V})$ is a hypersurface of degree $\delta$
in $\PP^{r-1}$. Here $Y = (y_1,y_2,\ldots,y_r)^T$ is the column vector of
coordinates on $\PP^{r-1}$.
Each of the $\binom{n}{r-1}$ twistor coordinates $[\, Z \,| \, Y \,]_{i_1 i_2 \cdots i_{r-1}}\,$
is a linear form in $y_1,y_2,\ldots,y_r$.

\begin{corollary} \label{cor:projhyper}
The equation of any hypersurface obtained by projecting
$\mathcal{V} \subset \PP^{n-1}$ into $\PP^{r-1}$ is read off from
the Chow form
${\rm C}_\mathcal{V}$ by replacing Pl\"ucker coordinates 
with linear forms via~(\ref{eq:twistor1}).
\end{corollary}

\begin{example}
Rational quartic curves in $\PP^2$ are images of the
curve $\mathcal{V}$ in Example \ref{ex:quarticcurve}
under  projections $Z: \PP^4 \dashrightarrow \PP^2$.
To compute these plane curves, we consider the $3 \times 6$ matrix
$$
[\, Z \,| \, Y \,] \,\, = \,\, \begin{small}
\begin{bmatrix}
\, z_{11} & z_{12} & z_{13} & z_{14} & z_{15} & \, y_1 \,\\
\, z_{21} & z_{22} & z_{23} & z_{24} & z_{25} & \, y_2 \, \\
\, z_{31} & z_{32} & z_{33} & z_{34} & z_{35} & \, y_3 \,
\end{bmatrix}. \end{small}
$$
The equation of the plane quartic curve $Z(\mathcal{V})$ is
obtained from the B\'ezout determinant  in (\ref{eq:bezout})
by replacing $p_{ij}$ with the $3 \times 3$ minor
of this $3 \times 6$ matrix having column indices $i,j,6$.

We can similarly project the Veronese surface (Example~\ref{ex:veronese1}) 
into $\PP^3$ with a $4 \times 6$ matrix $Z$.
The resulting rational quartics are
known as {\em Roman surfaces} or {\em Steiner surfaces}.
Algebraically, we 
substitute twistor coordinates into 
the Chow form in \cite[Section~2.2]{Cortona}.
\end{example}

The classical theory of Chow forms extends
naturally to a hierarchy of higher Chow forms,
which characterize linear spaces that are tangent to $\mathcal{V}$.
These are also known as  coisotropic hypersurfaces.
Their degrees are the polar degrees of $\mathcal{V}$,
as shown by Kohn in \cite{kohn}.
On the far end of the hierarchy is the
dual variety, which characterizes hyperplanes tangent to $\mathcal{V}$.
On the near end, right next to the Chow form, is the
Hurwitz form, which we now review.

Fix $\mathcal{V} \subset \PP^{n-1}$ of dimension $d$ and degree $\delta$.
For $L \in {\rm Gr}(n-d,n)$ generic, the intersection $\mathcal{V} \cap L$ consists of
$\delta$ distinct points. Let $\mathcal{H}_\mathcal{V}$ be the 
subvariety of ${\rm Gr}(n-d,n)$ consisting of all $L$ where this fails.
Geometrically, such $L $ are tangent to $\mathcal{V}$.
If $\delta \geq 2$ then $\mathcal{H}_\mathcal{V}$ is a hypersurface in ${\rm Gr}(n-d,n)$.
This hypersurface was studied in \cite{Hur}.
Its defining equation ${\rm H}_\mathcal{V}$ is  the
{\em Hurwitz form} of $\mathcal{V}$. 
We close this section by deriving  the Hurwitz analog to
Corollary~\ref{cor:projhyper}.

\begin{theorem} \label{thm:hurwitzprojection}
The equation of the branch locus of any projection of
$\mathcal{V} \subset \PP^{n-1}$ into $\,\PP^{d}$ is read off from
the Hurwitz form
${\rm H}_\mathcal{V}$ by replacing Pl\"ucker coordinates 
with linear forms via~(\ref{eq:twistor1}).
\end{theorem}

\begin{proof}
The branch locus consists of all points $y$ such that
the fiber $Z^{-1}(y)$ is tangent to $\mathcal{V}$ at some point.
This happens if and only if the Hurwitz form ${\rm H}_\mathcal{V}$
vanishes at the subspace $Z^{-1}(y)$. 
Evaluating the Hurwitz form in primal Pl\"ucker coordinates $p$ at  any such fiber
translates into the algebraic operation of replacing $p$ with  twistor coordinates,
by Proposition \ref{prop:chowprojection}.
\end{proof}

\begin{example}[Branch curve of the Veronese]
Let $Z$ be a general $3 \times 6$ matrix,
defining a projection $\PP^5 \dashrightarrow \PP^2$,
and let $\mathcal{V}$ be the Veronese surface in Example \ref{ex:veronese1}.
The Hurwitz form ${\rm H}_\mathcal{V}$ has degree six
in the Pl\"ucker coordinates. Its expansion in
primal Stiefel coordinates is the  {\em tact invariant} of two conics. The explicit 
formula in primal Pl\"ucker coordinates $p_{ij}$
is shown in \cite[Example 2.7]{Hur}.
The branch curve in $\PP^2$ has degree six. Its equation is 
obtained from ${\rm H}_\mathcal{V}$ by replacing $p_{ij} $ with
the $3 \times 3$ minor indexed by $i,j,7$
in the $3 \times 7 $ matrix $ [\, Z \,| \, Y \,]$.
\end{example}

\section{From Projective Space to the Grassmannian}
\label{sec:three}

In the previous section, we encoded a subvariety $\mathcal{V}$ in 
$\PP^{n-1} = {\rm Gr}(1,n)$ by a single equation ${\rm C}_\mathcal{V}$.
 We here replace the ambient projective space with an
arbitrary Grassmannian ${\rm Gr}(k,n)$.
Now the degree of $\mathcal{V}$ is no longer an integer but
it is a cohomology class $[\mathcal{V}]$. Recall that $H^*({\rm Gr}(k,n), \ZZ)$
is isomorphic to $\ZZ^{\binom{n}{k}}$, with basis given as follows.
The set $\binom{[n]}{k}$  of $k$-sets 
$I = \{i_1 < i_2 < \cdots < i_k \} $ in $[n] = \{1,\ldots,n\}$  is partially ordered by
setting $J \leq I$ if $j_1 \leq i_1$  and  $j_2 \leq i_2 $ and $\cdots$ and $j_k \leq i_k$.
For $I \in \binom{[n]}{k}$, the {\em Schubert variety}
$\mathcal{S}_I$ is defined by 
$q_J = 0$  for all $J$ that do not satisfy $J \geq I$. Equivalently, it consists of $k$-dimensional vector spaces which meet $\text{span}(e_1, \ldots, e_s)$ in dimension at least $i_s$ for each $s \leq n.$ The Schubert variety $\mathcal{S}_I$ is irreducible of dimension
$\sum_{s=1}^k (i_s-s)$. In particular, if
$I = \{1,\ldots,k-1,k\}$ then ${\rm dim}(\mathcal{S}_I) = 0$, so
$\mathcal{S}_I$ is a point.
The  Schubert classes $[\mathcal{S}_I]$ form a $\ZZ$-basis of
$H^*({\rm Gr}(k,n), \ZZ)$.

For any subvariety   $\mathcal{V}$, its cohomology class 
has a unique expansion into Schubert classes:
\begin{equation}
\label{eq:schubert} [\mathcal{V}] \,\,\, = \,\,\, \sum_I \delta_I( \mathcal{V}) \cdot [\mathcal{S}_I]. 
\end{equation}
Each coefficient $\delta_I(\mathcal{V})$ is a nonnegative integer, which is zero unless
$\,\sum_{s=1}^k (i_s - s) = {\rm dim}(\mathcal{V}) $.

There is a natural involution on $\binom{[n]}{k}$. Namely, the complement of $I$
is the index set
$$ I^c \,\, = \,\, \bigl\{\,n+1-i_k\,,\, n+1-i_{k-1},\ldots\,,\, n+1-i_1 \,\bigr\}. $$
The Schubert varieties $\mathcal{S}_I$ and $\mathcal{S}_{I^c}$ 
have complementary dimensions. For a general matrix $g \in {\rm GL}(n,\CC)$, 
and any index set $J$ that satisfies ${\rm dim}(\mathcal{S}_J) = {\rm dim}(\mathcal{S}_I) = 
{\rm codim}(\mathcal{S}_{I^c})$, we have
$$  g\, \mathcal{S}_J \, \,\cap \,\, \mathcal{S}_{I^C} \,\, = \,\,
\begin{cases}
\hbox{a point} & {\rm if} \,\,I = J, \\
\quad \,\emptyset & {\rm otherwise}.
\end{cases}
$$
This yields the following method for computing the 
cohomology class (\ref{eq:schubert}) from the ideal of~$\mathcal{V}$.

\begin{proposition} The coefficient $\delta_I(\mathcal{V})$ in the
class $[\mathcal{V}]$ of the subvariety $\mathcal{V} \subset {\rm Gr}(k,n)$ equals the number of
points in the intersection $  \,g\, \mathcal{V} \, \,\cap \,\, \mathcal{S}_{I^C} \,$
for a general matrix $g \in {\rm GL}(n,\CC)$.
\end{proposition}

We can thus compute the numbers $\delta_I(\mathcal{V})$ 
either symbolically (using Gr\"obner bases) or numerically
(using numerical algebraic geometry). The latter approach
rests on {\em Schubert witness sets} from homotopy continuation.
These were introduced by Sottile in \cite[Section~4]{Sot}.

\smallskip

We now come to our main topic, namely the Chow-Lam form. Recall that 
$\mathcal{V} \subset {\rm Gr}(k,n)$ was assumed to have dimension $ k(r-k)-1$ for some $r \in \{k+1,\ldots,n\}$.
We define $\mathcal{CL}_\mathcal{V}$ to be the subvariety of 
${\rm Gr}(k+n-r,n)$ which consists of all points $P$ such that
the inclusion  $Q \subseteq P$ holds for some $Q \in \mathcal{V}$. 
We call $\mathcal{CL}_\mathcal{V}$ the {\em Chow-Lam locus}
of the given variety $\mathcal{V}$.

\begin{lemma} \label{lem:CLproper} If $\mathcal{V}$ is irreducible then
 $\mathcal{CL}_\mathcal{V}$ is a proper irreducible subvariety of ${\rm Gr}(k+n-r,n)$.
 \end{lemma}
 
 \begin{proof}
 We argue that
  $\mathcal{CL}_\mathcal{V}$ is expected to be a hypersurface
in ${\rm Gr}(k+n-r,n)$. 
There are $k(r-k)-1$ degrees of freedom in
fixing a point $Q \in \mathcal{V}$. 
After fixing $Q$, we choose $P$. The variety of all 
subspaces $P \in {\rm Gr}(k+n-r,n)$
that contain $Q$ is a Grassmannian ${\rm Gr}(n-r,n-k)$.
So, there are
$(n-r)(r-k) = {\rm dim}({\rm Gr}(n-r,n-k))$ degrees of freedom for choosing $P$.
Our construction parametrizes an irreducible incidence variety whose dimension is
the sum
\begin{equation}
\label{eq:dof}
 k(r-k)- 1 \, \,+ \,\, (n-r)(r-k) \,\, = \,\, (r-k)(k+n-r) -1 \,\,= \,\, {\rm dim}({\rm Gr}(k+n-r,n)) - 1. \end{equation}
 The incidence variety parametrizes all pairs $(Q,P)$ as above.
  It is irreducible and maps onto the Chow-Lam locus $\mathcal{CL}_\mathcal{V}$.
 This shows that $\mathcal{CL}_\mathcal{V}$ is irreducible of dimension at most (\ref{eq:dof}).
  \end{proof}

 The {\em Chow-Lam form} of $\mathcal{V}$ is the
 polynomial ${\rm CL}_\mathcal{V} $ that defines the Chow-Lam locus
 $\mathcal{CL}_\mathcal{V}$, provided this has codimension one.
 Otherwise, $\mathcal{V}$ is  {\em degenerate} and we set  ${\rm CL}_\mathcal{V} = 1$.
 This definition is analogous to the definition of a degenerate dual variety and discriminant (cf.~\cite{GKZ}).
 For any projective variety, the dual variety is expected to be a hypersurface,
 and its defining polynomial is the discriminant. However, it can happen
 that the dual variety has codimension $\geq 2$, in which case the discriminant
 is $1$. We note that ${\rm CL}_\mathcal{V}$ is unique 
 up to scaling and modulo the Pl\"ucker relations for ${\rm Gr}(k+n-r,n)$.
We can write it either in Pl\"ucker coordinates
(primal or dual) or in Stiefel coordinates (primal or~dual). 

\begin{remark}
In our definition of the Chow-Lam locus it is assumed
that the variety $\mathcal{V}$ has dimension $k(r-k)-1$.
This is, of course, a  restrictive hypothesis.
Our rationale for this is
that we would like $\mathcal{CL}_\mathcal{V}$ to have
codimension one. Giving up this hypothesis would
take us to the general setting
of higher Chow-Lam loci, which concerns
tangencies between $\mathcal{V}$
and arbitrary subGrassmannians ${\rm Gr}(k,L)$ of ${\rm Gr}(k,n)$.
This is discussed at the end of Section~\ref{sec5}.
\end{remark}

The Chow-Lam form is named for Thomas Lam, who studies universal projections of positroid varieties in 
\cite[Chapter 18]{LamCurrent}. We will establish the relationship between universal projections and the Chow-Lam form in Proposition \ref{prop:chowlamprojection}. Lam also computes cohomology classes of projections in 
\cite[Proposition 3.5]{LamStanley}. This is fundamental for  Theorem \ref{thm:CL} below.

We next explore the issue of degeneracy for 
Schubert varieties in a small Grassmannian.

\begin{example}[$k=2,n=5,r=4$] \label{ex:254}
Let $\mathcal{V}$ be a variety of dimension $3$ in the
$6$-dimensional Grassmannian ${\rm Gr}(2,5)$.
Here, $k+n-r=3$, so the Chow-Lam locus $\mathcal{CL}_\mathcal{V}$ is
a subvariety of ${\rm Gr}(3,5)$. We shall use
dual Pl\"ucker coordinates $q$ on ${\rm Gr}(2,5)$
and primal Pl\"ucker coordinates $p$ on ${\rm Gr}(3,5)$.
The inclusion $Q  \subset P$ is expressed algebraically by the matrix equation
	\begin{equation} \label{eq:twoskewsymmetric}
\begin{bmatrix}
			0 & p_{12} & p_{13} & p_{14} & p_{15}\\
			-p_{12} & 0 & p_{23} & p_{24} & p_{25}\\
			-p_{13} & -p_{23} & 0 & p_{34} & p_{35}\\
			-p_{14} & -p_{24} & -p_{34} & 0 & p_{45}\\
			-p_{15} & -p_{25} & -p_{35} & -p_{45} & 0
		\end{bmatrix} \cdot 
		\begin{bmatrix}
			0 & q_{12} & q_{13} & q_{14} & q_{15}\\
			-q_{12} & 0 & q_{23} & q_{24} & q_{25}\\
			-q_{13} & -q_{23} & 0 & q_{34} & q_{35}\\
			-q_{14} & -q_{24} & -q_{34} & 0 & q_{45}\\
			-q_{15} & -q_{25} & -q_{35} & -q_{45} & 0
		\end{bmatrix} \,\,=\,\, 0.
	\end{equation}
This is a system of $25$ bilinear equations in $(p,q)$.
We augment this
by the equations in $q$ that define $\mathcal{V}$, we 
saturate with respect to the $q$-variables, and we finally eliminate all $q$-variables.
The resulting ideal in the ten $p$-variables defines the Chow-Lam locus 
  $\mathcal{CL}_\mathcal{V} \subset {\rm Gr}(3,5)$.
  
  First suppose that $\mathcal{V}$ is defined by three
  general linear forms in $p$. Then $\mathcal{CL}_\mathcal{V}$
  has codimension one, and the Chow-Lam ${\rm CL}_\mathcal{V}$ form has degree two.
  For a concrete example, 
  take  $\mathcal{V} = V(q_{12}+q_{13}, q_{24}+q_{25}, q_{23}+q_{35})$. 
  The threefold $\mathcal{V}$ has degree five in  Pl\"ucker space $\PP^9$. It
  parametrizes all lines in $\PP^4$ that meet three given planes in $\PP^4$.
  The algorithm above yields
  $$ {\rm CL}_\mathcal{V} \,\, = \,\,
  p_{14}\,(p_{24}-p_{25}-p_{34}+p_{35})\,   + \,(p_{12}-p_{14}+p_{15})\,p_{45}.
  $$
  Next, we consider the two Schubert varieties of dimension $3$ in ${\rm Gr}(2,5)$. They are
  $$ \mathcal{S}_{24} \,\, = \,\, V(q_{12},q_{13},q_{14},q_{15},q_{23}) \quad {\rm and} \quad
  \mathcal{S}_{15} \,\, = \,\, V(q_{12},q_{13},q_{14},q_{23},q_{24},q_{34}). $$
  The first Schubert threefold is non-degenerate and its Chow-Lam form equals
  ${\rm CL}_{\mathcal{S}_{24}} = p_{45}$.
  The second Schubert threefold $\mathcal{S}_{15}$ is
  found to be degenerate. Our algorithm reveals that
  $$ \mathcal{CL}_{\mathcal{S}_{15}} \,\, = \,\, V(p_{15},p_{25},p_{35}, p_{45}). $$
  This Chow-Lam locus   has codimension two in ${\rm Gr}(3,5)$,
  and hence ${\rm CL}_{\mathcal{S}_{24}} = 1$.
\end{example}

We now come to the first result of this section. It will explain
the findings in Example~\ref{ex:254}.
Let $I = \{r-k, r-k+2, r-k+3, ...,r\}$. This is the
unique index set $I = \{i_1,\ldots,i_k\}$ of the correct codimension which satisfies $i_k = r$.
We set $\lambda(\mathcal{V}) := \delta_I(\mathcal{V})$, and we call this integer
 the {\em Chow-Lam degree}
of the given variety $\mathcal{V}$. Note that
the Chow-Lam degrees of the  three varieties
in Example \ref{ex:254} are
$\lambda(\mathcal{V}) = 2$,
$\lambda(\mathcal{S}_{24}) = 1$, and
$\lambda(\mathcal{S}_{15}) = 0$.

\begin{theorem} \label{thm:CL}
Let $\mathcal{V}$  be a subvariety of dimension $k(r-k)-1$
in the Grassmannian ${\rm Gr}(k,n)$.
The Chow-Lam form ${\rm CL}_\mathcal{V}$ is a polynomial of
degree $\lambda(\mathcal{V})$ in the Pl\"ucker coordinates on ${\rm Gr}(k{+}n{-}r,n)$.
In particular, $\mathcal{V}$ is degenerate if and only if 
its Chow-Lam degree $\lambda(\mathcal{V}) $ is~zero.
\end{theorem}

This formula for the degree of ${\rm CL}_\mathcal{V}$ generalizes 
Corollary \ref{cor:chowdegree}. 
Indeed, if $k=1$ then our variety $\mathcal{V}$ has dimension
$d = r-2$ in $\PP^{n-1} = {\rm Gr}(1,n)$, and its cohomology class  (\ref{eq:schubert}) is
$$ [\mathcal{V}] \, = \,{\rm degree}(\mathcal{V})  \cdot [\mathcal{S}_{\{n-d\}}]. $$
Here, the Chow-Lam degree is just the classical degree, i.e.~$\,\lambda(\mathcal{V}) = {\rm degree}(\mathcal{V})$. 
The argument that led to Corollary \ref{cor:chowdegree}
can be generalized to $k \geq 2$. Our proof of Theorem \ref{thm:CL}
will be based on this. We begin with the generalization of the
intersection formula  in Proposition~\ref{prop:chowintersec}.

For a subspace $L$ of dimension $\ell$ in $\CC^n$, we introduce the relative Grassmannian
$$ {\rm Gr}(k,L) \,\, = \,\, \bigl\{ P \in {\rm Gr}(k,n) \,|\, P \subseteq L \,\bigr\}
\quad \simeq \,\,\, {\rm Gr}(k,\ell). $$
With this notation, the definition of the Chow-Lam locus can be rewritten as follows:
\begin{equation}
\label{eq:CLdef2} \mathcal{CL}_\mathcal{V} \,\, = \,\, 
\bigl\{ \, L \in {\rm Gr}(k+n-r,n) \,: \, \mathcal{V} \,\cap \, {\rm Gr}(k,L) \,\not= \,\emptyset \, \bigr\}. 
\end{equation}

\begin{lemma}[Intersection Formula] \label{lem:chowlamintersec}
Let $L$ and $M$ be linear subspaces of $\PP^{n-1}$
such that ${\rm codim}(L \cap M) =
{\rm codim}(L) + {\rm codim}(M) = r-k$. Then
\begin{equation}
\label{eq:VLM}
{\rm CL}_{\mathcal{V}\, \cap \,{\rm Gr}(k,L)} (M) \,\, = \,\, {\rm CL}_{\mathcal{V}}(L \cap M). 
\end{equation}
\end{lemma}

\begin{proof} This follows from (\ref{eq:CLdef2}) and the identity
${\rm Gr}(k,L \cap M) \, = \, {\rm Gr}(k,L) \, \cap \, {\rm Gr}(k,M)$.
\end{proof}

\begin{corollary} \label{cor:godown1}
The Chow-Lam forms ${\rm CL}_\mathcal{V}$ and
${\rm CL}_{\mathcal{V}\, \cap \,{\rm Gr}(k,L)} $ have the same
degree, when written in their respective Pl\"ucker coordinates. 
\end{corollary}

The following lemma is based on a standard argument in Schubert calculus.

\begin{lemma}
Let $L$ be a general subspace of codimension $\rho$ in $\CC^n$.
The  intersection $ \mathcal{S}_I \cap {\rm Gr}(k,L) $ is non-empty
if and only if $i_1 > \rho$. In this case, its class is a Schubert class, namely
$$ [\mathcal{S}_I \cap {\rm Gr}(k,L)] \,\, = \,\,
[\mathcal{S}_{I - (\rho,\rho,\ldots,\rho)}]. $$
\end{lemma}

\begin{corollary} \label{cor:godown2}
Using notation as above, we have
$\,\lambda(\mathcal{V}) = 
\lambda(\mathcal{V} \, \cap \,{\rm Gr}(k,L) )$.
\end{corollary}

\begin{proof}[Proof of Theorem \ref{thm:CL}]
By Corollaries \ref{cor:godown1} and \ref{cor:godown2},
the assertion follows by induction on $n-r$. It suffices
  to prove the assertion for the base case $n-r = 0$, where
   $\mathcal{V} = \mathcal{CL}_\mathcal{V}$ is a hypersurface in ${\rm Gr}(k,n)$.
Note that $[\mathcal{V}] = \lambda(\mathcal{V}) \cdot [\mathcal{S}_I]$
where $I = \{n-k,n-k+2,n-k+3,\ldots,n\}$. This
$[\mathcal{S}_I]$ is the class of a hyperplane section of ${\rm Gr}(k,n)$
inside its Pl\"ucker embedding. Hence $\mathcal{V}$ is
defined by a polynomial of degree $\lambda(\mathcal{V})$
 in the $\binom{n}{k}$ Pl\"ucker coordinates, which
 is unique modulo Pl\"ucker relations. Moreover, this polynomial
equals the Chow-Lam form ${\rm CL}_{\mathcal{V}}$, since
 $\mathcal{V} = \mathcal{CL}_\mathcal{V}$.
 \end{proof}

We now come to the projection formula for $k \geq 2$, which is the direct generalization of
Proposition \ref{prop:chowprojection}.
Let $Z$ be a general $r \times n$ matrix as in Section 2.
Then $Z$ defines a projection $\PP^{n-1} \dashrightarrow \PP^{r-1}$
and this induces a rational map
${\rm Gr}(k,n) \dashrightarrow {\rm Gr}(k,r)$.
We write $Z(\mathcal{V})$ for the closure  in ${\rm Gr}(k,r)$ of the image of the
variety $\mathcal{V} \subset {\rm Gr}(k,n)$ under the map induced by~$Z$.

\begin{proposition}[Projection Formula] \label{prop:chowlamprojection}
The Chow-Lam form ${\rm CL}_{Z (\mathcal{V})}$ is obtained from the
Chow-Lam form ${\rm CL}_{\mathcal{V}}$ by replacing
the primal Pl\"ucker coordinates with twistor coordinates:
\begin{equation}
\label{eq:twistor2}
  \qquad p_{i_1 i_2 \cdots i_{r-k}} \,\, = \,\, \,[\, Z \,| \, Y \,]_{i_1 i_2 \cdots i_{r-k}}\quad
{\rm for} \,\, 1 \leq i_1 < i_2 < \cdots < i_{r-k} \leq n. 
\end{equation}
This expresses ${\rm CL}_{Z (\mathcal{V})}$ in 
dual Stiefel coordinates on ${\rm Gr}(k,r)$,
given by the $r \times k$ matrix $Y$.
\end{proposition}

\begin{proof}
We identify the $r \times k$ matrix $Y$ with its column span,
which is a $(k-1)$-dimensional subspace in $\PP^{r-1}$,
or a $k$-dimensional linear subspace in $\CC^r$.
The inverse image $Z^{-1}(Y)$ is a projective subspace of dimension
$n-r+k-1$ in $\PP^{n-1}$, or
a linear subspace of dimension $n-r+k$ in $\CC^n$.
The subspace $Y$ is a point in $\mathcal{CL}_{Z (\mathcal{V})}$ if and only if
$Z^{-1}(Y)$ lies in  $\mathcal{CL}_{\mathcal{V}}$.
\end{proof}

\begin{corollary}\label{cor:CLprojhyper}
	Any hypersurface obtained by projecting $\mathcal{V} \in {\rm Gr}(k,n)$ into ${\rm Gr}(k,r)$ is read off from
	 ${\rm CL}_{\mathcal V}$ by replacing primal Pl\"ucker coordinates with twistor coordinates via
	(\ref{eq:twistor2}).
\end{corollary}

\begin{example}[Genus 6 geometry] \label{ex:fano}
Fix $k=2$ and $n=5$. We start with the case~$r=4$. Let
$\mathcal{V}$ be the subvariety of ${\rm Gr}(2,5)$ defined by
two general linear forms and one general quadric  in the $10$ dual Pl\"ucker coordinates $q_{ij}$.
Then $\mathcal{V}$ is a {\em Fano threefold} of genus $6$, see \cite{Log}.
This threefold has degree $10$ in the ambient space $\PP^9$.
Its Chow-Lam degree is $\lambda(\mathcal{V}) = 4$.
We compute ${\rm CL}_\mathcal{V}$ 
as in Example \ref{ex:254}, i.e.~we augment the
ideal of $\mathcal{V}$ by the bilinear equations in (\ref{eq:twoskewsymmetric})
and then eliminate the $q$-variables after saturation.
The result is a quartic in primal Pl\"ucker coordinates $p_{ij}$
on ${\rm Gr}(3,5)$. This is the Chow-Lam form ${\rm CL}_\mathcal{V}$ 
of our Fano threefold~$\mathcal{V}$.

Consider the projection ${\rm Gr}(2,5) \dashrightarrow {\rm Gr}(2,4)$
given by a general $4 \times 5$ matrix $Z$. The image
$Z(\mathcal{V})$ is a hypersurface in ${\rm Gr}(2,4)$.
Its defining polynomial ${\rm CL}_{Z(\mathcal{V})}$ is obtained by
setting $ p_{i j} =  [\, Z \,| \, Y \,]_{ij} $. These twistor coordinates
are the $4 \times 4$ minors of the $4 \times 7$ matrix 
$[\, Z \,| \, Y \,]$ which use the last two columns.
Hence ${\rm CL}_{Z(\mathcal{V})}$ has degree $4$ in the $2 \times 2$ minors of~$Y$.

 We next assume that $r=3$ and
 $\mathcal{V} \subset {\rm Gr}(2,5)$ is cut out by
four general linear forms and one general quadric  in the  $q_{ij}$.
Here $\mathcal{V}$ is a {\em canonical curve} of genus $6$ and degree~$10$ in the $\PP^5$ defined by the four linear forms;
see \cite[Lemma 4.1]{GHSV}. We have $\lambda(\mathcal{V}) = 10$, so the Chow-Lam form
${\rm CL}_\mathcal{V}$ is a polynomial of degree $10$
in the coordinates $(p_1,\ldots,p_5)$ on ${\rm Gr}(4,5) \simeq (\PP^4)^\vee$.

Consider the projection ${\rm Gr}(2,5) \dashrightarrow {\rm Gr}(2,3)$
given by a general $3 \times 5$ matrix $Z$. The image
$Z(\mathcal{V})$ is a singular curve of degree $10$ in the
dual projective plane ${\rm Gr}(2,3) = (\PP^2)^\vee$,
with coordinates $(y_1,y_2,y_3)$.
We find its equation by the substitution in
(\ref{eq:twistorex}) with $i=1,2,3,4,5$.
\end{example}

The join of two projective varieties is an important operation in
algebraic geometry. The construction was characterized at
the level of Chow forms in \cite[Section 4.1]{DS}.
 The definition of the join generalizes to subvarieties
 of Grassmannians and their Chow-Lam forms.

Let $\mathcal{V}$ and $\mathcal{W}$ be disjoint subvarieties of ${\rm Gr}(k,n)$.
		We define the
		 {\em embedded join} of $\,\mathcal{V}$ and $\mathcal{W}$ to be the union
		\[
		J(\mathcal{V}, \mathcal{W}) \quad := \quad
		\bigcup_{V \in \mathcal{V}, W \in \mathcal{W}} {\rm Gr}(k, V + W).\]
This is a subvariety in ${\rm Gr}(k,n)$.
The dimension of $J(\mathcal{V}, \mathcal{W})$ equals
$\dim \mathcal{V} + \dim \mathcal{W} + k^2$, provided
this number does not exceed $k(n-k)$.
The following result is dual to Lemma \ref{lem:chowlamintersec}.

	\begin{proposition}\label{prop:kjoin}
	Fix $\mathcal{V} \subset {\rm Gr}(k,n)$ of dimension $k(r-k)-1$,
	and consider general subspaces $L,M \subset \CC^n$ with
	${\rm codim}(L+M) = r-k$.
		We have the equality of Chow-Lam forms
		\begin{equation}
	{\rm CL}_{J(\mathcal V, {\rm Gr}(k,L))}(M) \,\, = \,\,{\rm CL}_{\mathcal V}(L+ M).
		\end{equation}
	\end{proposition}

The numerology of dimensions matches since  ${\rm dim} \,J(\mathcal V, {\rm Gr}(k,L)) = k(r+{\rm dim}(L)-k)-1$.
We omit the proof, and instead close with an example to illustrate the embedded join.

	\begin{example}[Joining a curve]
	Let $L \in {\rm Gr}(3,6)$ and fix a curve $\mathcal{V} \subset {\rm Gr}(2,6)$.	Then the join of $\mathcal{V}$ and ${\rm Gr}(2,L) \simeq \PP^2$
is a hypersurface in ${\rm Gr}(2,6)$. Its equation is obtained
by writing ${\rm CL}_\mathcal{V}$ in
dual Pl\"ucker coordinates for $L+M$. 
For instance, if $\mathcal{V}$ is a 
		genus $8$ canonical curve  \cite[Section 4]{GHSV}
		then $\lambda(\mathcal{V}) = 14$,
		and therefore ${\rm CL}_{J(\mathcal V, {\rm Gr}(k,L))}$
		is an equation of degree $14$.
	\end{example}

\section{Matroids and Positroids} \label{sec4}

In this section we turn to subvarieties of
${\rm Gr}(k,n)$ that are defined by the
vanishing of Pl\"ucker coordinates.
For a point $\xi \in {\rm Gr}(k,n)$ in dual Pl\"ucker coordinates,
the associated {\em matroid} $M = M_\xi$
is the set of all indices $I \in \binom{[n]}{k}$
such that $\xi_I = 0$. Thus $M$
is a matroid of rank $k$ on $[n]$,
given by its list of  {\em nonbases}. The
nonbases are the indices of 
dual Pl\"ucker coordinates that are zero.
The matroid $M$ is a {\em positroid}
if $M=M_\xi$ for some $\xi$ with dual Pl\"ucker coordinates all non-negative real numbers.

Conversely, for  a matroid $M$ of rank $k$ on $[n]$, the {\em realization space}
is the constructible set $\{ \, \xi \in {\rm Gr}(k,n) \,: \, M_\xi = M \}$.
The Zariski closure of the realization space in  
the Grassmannian ${\rm Gr}(k,n)$ is denoted  $\mathcal{V}_M$ 
and called the {\em matroid variety} of $M$. Positroids are of
special interest at the interface of combinatorics and physics,
and one uses the term {\em positroid variety} for $\mathcal{V}_M$
if $M$ happens to be a positroid.
Positroid varieties behave much better than general matroid varieties;
as shown in \cite{KLS}.
We seek to compute the Chow-Lam forms of these  varieties.
The images $Z(\mathcal{V}_M)$ of positroid varieties
$\mathcal{V}_M$ under projections $Z : {\rm Gr}(k,n)
\dashrightarrow {\rm Gr}(k,r)$ are known as 
{\em amplituhedron varieties}. Lam's work in
\cite{LamCurrent} aims at computing  the equations defining $Z(\mathcal{V}_M)$ in
twistor coordinates, using the Projection Formula
for the Chow-Lam form.

\begin{example}[$k=2,n=6,r=5$] \label{ex:positroid265}
The matroid $M = \{12,34,56\}$ is a positroid. Its positroid variety 
$\mathcal{V}_M = V(\xi_{12},\xi_{34},\xi_{56})$ is
a subvariety of dimension $ k(r-k) -1 = 5$ in the $8$-dimensional Grassmannian ${\rm Gr}(2,6)$.
This variety has Chow-Lam degree $\lambda(\mathcal{V}_M) = 2$.
 Its Chow-Lam form is the condition for the three lines $\overline{12}$,
 $\overline{34}$ and $\overline{56}$ in $\PP^2$ to be concurrent:
\begin{equation}
\label{eq:CLmatroid1}
  {\rm CL}_{\mathcal{V}_M} \,\, = \,\, 
p_{123} p_{456} \,  - \, p_{124} p_{356}. 
\end{equation}
The derivation of this formula is similar to (but easier than) that in Example~\ref{ex:positroid92}.
The amplituhedron variety $Z(\mathcal{V}_M)$ is a hypersurface in
${\rm Gr}(2,5)$. Its equation is found by replacing the $20$ $\,p_{ijk}$ with the
$5 \times 5$ minors of the $5 \times 7$ matrix $\,[ \, Z \,| \, Y \,]$ that use the last two columns.
\end{example}

\begin{remark}[Schubert matroids]
Among the matroid varieties $\mathcal{V}_M$ are
the Schubert varieties $\mathcal{S}_I$. 
The associated matroids $M_I$ are very special. 
The bases of $M_I$ are the  index sets $J$ such that $J \geq I$.
Oh \cite{Oh} characterizes positroids
in terms of such {\em Schubert matroids}. Namely, positroid varieties $\mathcal{V}_M$
are intersections of  cyclically shifted Schubert varieties $\mathcal{S}_I$.
\end{remark}

Every matroid variety $\mathcal{V}_M$ is naturally stratified
into toric varieties. The torus $T = (\CC^*)^n$ acts
on the Grassmannian ${\rm Gr}(k,n)$. Given any
point $\xi \in {\rm Gr}(k,n)$, we write $\mathcal{T}_\xi$  for 
the Zariski closure of the orbit $T \xi$ in ${\rm Gr}(k,n)$.
This is a toric variety of dimension $\leq n-1$, and the
dimension equals $n-1$ when $\xi$ is generic.
Up to a multiplicative change of coordinates, the
toric varieties $\mathcal{T}_\xi$ depend only on the matroid
$M = M_\xi$, and we therefore write $\mathcal{T}_M = \mathcal{T}_\xi$.

We call $\mathcal{T}_M$ the toric variety associated with the matroid $M$.
The corresponding polytope is the {\em matroid polytope} of $M$.
For example, if $\xi$ is a generic point in ${\rm Gr}(k,n)$ then
$M$ is the uniform matroid, and the
matroid polytope is the hypersimplex $\Delta(k,n)$. Here,
 $\mathcal{T}_M$ is the $(n-1)$-dimensional toric variety
 parametrized by all $\binom{n}{k}$
squarefree monomials of degree $k$ in $n$ unknowns.
It is interesting to study the
Chow-Lam forms of the toric varieties $\mathcal{T}_M$.

\begin{example}[Hypersimplex]
We fix a general $2 \times 6$ matrix $(a_{ij})$  with row span $\xi \in {\rm Gr}(2,6)$.
The toric variety $\mathcal{T}_\xi$ has
dimension $5$ and degree $26$ in  ${\rm Gr}(2,6) \subset \PP^{14}$. Its
 prime ideal is generated by $30$ quadrics like
 $ \, \xi_{13} \xi_{24} \,q_{12}  q_{34} - \xi_{12} \xi_{34} \,q_{13} q_{24}$.
 The Chow-Lam form is found to~be
 \begin{equation}
\label{eq:CLmatroid2}
 \begin{matrix} {\rm CL}_{\mathcal{T}_\xi} & = &
(\xi_{14} \xi_{26} \xi_{35} + \xi_{15} \xi_{24} \xi_{36}
           \,-\,  \xi_{16} \xi_{24} \xi_{35})\, q_{123} q_{456}
\,-\,  \xi_{13} \xi_{25} \xi_{46} \, q_{124} q_{356} \\ & & 
\,+\,  \xi_{12} \xi_{35} \xi_{46} \, q_{134} q_{256}
\,-\,  \xi_{12} \xi_{34} \xi_{56} \, q_{135} q_{246}
\,+\,  \xi_{13} \xi_{24} \xi_{56} \, q_{125} q_{346}.
\end{matrix}
\end{equation}
Modulo the Pl\"ucker ideal, $ {\rm CL}_{\mathcal{T}_\xi} $ is
the unique relation among the $3 \times 3$ minors of the matrix
$$
\begin{bmatrix}
a_{11} x_1 & a_{12} x_2 & a_{13} x_3 & a_{14} x_4 & a_{15} x_5 & a_{16} x_6 \\
a_{21} x_1 & a_{22} x_2 & a_{23} x_3 & a_{24} x_4 & a_{25} x_5 & a_{26} x_6 \\
y_1 & y_2 & y_3 & y_4 & y_5 & y_6 
\end{bmatrix}.
$$
Indeed, this matrix
is a parametric representation of $\mathcal{CL}_{\mathcal{T}_\xi}$,
with $x_i$ and $y_j$ as parameters.
This is a subvariety of dimension $8$ in the
Grassmannian ${\rm Gr}(3,6)$.
Note that (\ref{eq:CLmatroid1}) is obtained
from $(\ref{eq:CLmatroid2}) $ by setting
$\xi_{12} = \xi_{34} = \xi_{56} = 0$
and $\xi_{ij} = 1$ for all other $i,j$.
This reflects the fact that 
$\mathcal{T}_{M} = \mathcal{V}_{M}$
for the positroid $M$ in Example  \ref{ex:positroid265},
and that ${\rm CL}_{\mathcal{T}_\xi}$ specializes to
${\rm CL}_{\mathcal{T}_{M}}$.
\end{example}

We next focus on matroids of rank $k=2$.
In this case, every matroid $M$ is a positroid after relabeling.
We assume that $M$ is loopless, i.e.~for every $i$ there exists $j$ such that $\xi_{ij} \not= 0$.
With this mild hypothesis, every matroid is encoded by
a partition $\beta = (\beta_1 \geq \beta_2 \geq \cdots \geq \beta_t)$
of the integer $n$ into $t \geq 2$ parts.
The non-bases are the pairs that are contained in one of the $t$ blocks
$\{1,\ldots,\beta_1\}$, $\{\beta_1+1,\ldots,\beta_2\}$,
$\ldots\,$, $\{\beta_{t-1}+1,\ldots,\beta_t\}$ in the
corresponding set partition of $[n]$.
From now on, we identify rank $2$ matroids with rank $2$ positroids,
we encode them by partitions $\beta$,
and we write $\mathcal{V}_\beta$ for the associated positroid variety in ${\rm Gr}(2,n)$.
We note that
${\rm codim}(\mathcal{V}_\beta) = n-t$ and $r = \frac{n+t+1}{2}$.
In particular, $\,\mathcal{V}_{(1,1,1,\ldots,1)} = {\rm Gr}(2,n)\,$
and $\,\mathcal{V}_{(2,1,\ldots,1)} = \mathcal{S}_{n-2,n}$.

\begin{proposition}
Let $\beta$ be a partition of $n$ with $t$ parts,
and fix the univariate polynomial
 $$ f(x) \,\,=\,\, \prod_{i=1}^t \sum_{j=1}^{\beta_i} x^{j-1} . $$
 Suppose that $n-t$ is odd. 
                 Then the Chow-Lam degree of the positroid variety $\mathcal{V}_\beta$
                 is the coefficient of the monomial $\,x^{(n-t-1)/2}\,$ in $\,f(x)\,$ minus
                 the coefficient of $\,x^{(n-t-3)/2}\,$ in $\,f(x)$.
\end{proposition}

\begin{proof}
We represent Schubert classes by symmetric polynomials.
The variety $\mathcal{V}_\beta$ is the 
intersection of Schubert varieties of the form
$\mathcal{S}_{n-\beta_i,n}$ for $i=1,2,\ldots,t$. The
class $[\mathcal{S}_{n-\beta_i,n}]$ 
can be represented by the
complete homogeneous symmetric polynomial $h_{\beta_i-1}(x,y)$
of degree $\beta_i-1$ in two variables $x$ and $y$.
Hence $[\mathcal{V}_\beta]$ is the
product of these symmetric polynomials:
\[[\mathcal{V}_\beta]\,\, = \,\,h_{\beta_1-1}(x,y) \,h_{\beta_2-1}(x,y) \,\cdots \,h_{\beta_t-1}(x,y).\]
We write this product as
a sum of Schubert classes 
$\,x^i y^{n-t-i} + x^{i+1} y^{n-t-i-1} + \cdots + x^{n-t-i} y^i$
for $i = 0,1,\ldots,(n{-}t{-}1)/2$.
The Chow-Lam degree is the coefficient of the middle Schubert class
$\,[\mathcal{S}_{r-2,r}] = \,[\mathcal{S}_{\frac{n+t-3}{2},\frac{n+t+1}{2}}]$ in this decomposition. We dehomogenize it
by setting $y=1$. Thereafter we 
 conclude by observing that $x^{(n-t-1)/2}$ occurs once in each class $[\mathcal{S}_I]$ of codimension $n-t,$ whereas $x^{(n-t-3)/2}$ occurs once in each class except $[\mathcal{S}_{r-2,r}].$
\end{proof}

\begin{example} \label{ex:catalan}
For $n \leq 8$ all Chow-Lam degrees $\lambda(\mathcal{V}_\beta)$ are
$0$, $1$ or $2$. For every integer $n$ between $9$ and $30$,
the maximal value of $\lambda(\mathcal{V}_\beta)$
is attained by a unique partition $\beta$. In the
following table we list all triples $\,n,\lambda(\mathcal{V}_\beta), \beta\,$
that give the maxima for $n$ in that range:
\begin{footnotesize} $$
\begin{matrix}
9, 3, (3222) \qquad
10, 5, (22222) \qquad
11, 5, (222221) \qquad
12, 6, (33222) \qquad
13, 9, (322222) \qquad \\
14, 14, (2222222) \quad
15, 14, (22222221) \quad
16, 19, (3322222) \quad
17, 28, (32222222) \quad
18, 42, (222222222) \quad \\
19, 43, (33322222) \qquad
20, 62, (332222222) \qquad
21, 90, (3222222222) \qquad
22, 132, (22222222222) \qquad \\
23, 145, (3332222222) \quad
24, 207, (33222222222) \quad
25, 297, (322222222222) \quad 
26, 429, (2222222222222) \quad \\
27, 497, \! (333222222222) \quad 
28, 704, \! (3322222222222) \quad 
29, 1001, \! (32222222222222) \quad 
30, 1430, \! (222222222222222)
\end{matrix}
$$
\end{footnotesize}
The first entry says that the $9$-dimensional positroid variety
$\mathcal{V}_{(3222)} \subset {\rm Gr}(2,9)$ has Chow-Lam degree $3$. 
The cubic Chow-Lam form ${\rm CL}_{\mathcal{V}_\beta}$
was computed in Example \ref{ex:positroid92}.
One thing we notice in our table is the appearance of
Catalan numbers whenever $\,n\, \equiv \,2 \, \, {\rm mod} \,4 $.
This calls for an explanation. We shall provide one, along with
a new tool for computing  ${\rm CL}_{\mathcal{V}_\beta}$.
\end{example}

The degree of the Grassmannian ${\rm Gr}(2,s) \subset \PP^{\binom{s}{2}-1}$ is the Catalan number
$$ C_{s-2} \,\,=\,\, \frac{1}{s-1} \binom{2s-4}{s-2}. $$
Hence the Chow form ${\rm C}_{{\rm Gr}(2,s)}$ has degree $C_{s-2}$
in Pl\"ucker coordinates, and it has degree
$(2s-3) \cdot C_{s-2}$ in primal Stiefel coordinates.
This Chow form is the tool we promised above.

\begin{theorem} \label{thm:catalan}
Fix the partition $\beta = (2,2,2,\ldots,2)$ of
$n = 2t$ where $t=2s-3$ is odd. The Chow-Lam degree
of the positroid variety $\mathcal{V}_\beta$
is the Catalan number $C_{s-1} = \frac{n}{s} \,C_{s-2}$.
The Chow-Lam form ${\rm CL}_{\mathcal{V}_\beta}$
has degree $\,s\, C_{s-1} = n \, C_{s-2}\,$ in 
dual Stiefel coordinates, given by an $s {\times} n$ matrix
$X \,= [\,x_1 \,\, x_2 \,\, \cdots \,\,x_n\,]$.
We  obtain ${\rm CL}_{\mathcal{V}_\beta} $ from the Chow form 
 of the Grassmannian ${\rm Gr}(2,s)$
by evaluating the primal Stiefel coordinates at the columns of  the $\binom{s}{2} \times t$~matrix
$$ \tilde X \,\, = \,\, \big[\, x_1 \wedge x_2 \quad x_3 \wedge x_4 \quad
x_5  \wedge  x_6  \quad \cdots \quad
x_{n-1} \wedge x_n \,\bigr]. $$
\end{theorem}

This theorem explains the Catalan numbers in Example \ref{ex:catalan},
since these are the degrees of  the Grassmannians
${\rm Gr}(2,s)$.
 Before proving Theorem \ref{thm:catalan},
 we go over the case $s=4$.
   
 \begin{example}[Five lines admitting a transversal] \label{ex:transversal}
 Let $\beta = (2,2,2,2,2)$ and fix a $4 \times 10$ matrix
 $X = [\,x_1\,\,x_2\,\,\cdots\,\,x_9 \,\,x_0\,]$.
  Its columns $x_i$ are viewed as points in $\PP^3$.
  We are interested in the following condition on $X$:
  there exists a line $L$ which intersects the
  five lines $
  \overline{x_1x_2}$,
$\overline{x_3x_4}$,
$ \overline{x_5x_6}$,
$ \overline{x_7x_8}$,
$ \overline{x_9x_0}$.
This codimension $1$ condition is given by the Chow-Lam form of~$\mathcal{V}_\beta$:
\begin{equation}
\label{eq:AITformula} {\rm CL}_{\mathcal{V}_\beta} \,\,\, = \,\,\, {\rm det} \,\begin{bmatrix} 
    0  & q_{1234} & q_{1256} & q_{1278} & q_{1290}  \\
 q_{1234} &    0  & q_{3456} & q_{3478} & q_{3490}  \\
 q_{1256} & q_{3456} &    0  & q_{5678} & q_{5690} \\
 q_{1278} & q_{3478} & q_{5678} &    0  & q_{7890}  \\
 q_{1290} & q_{3490} & q_{5690} & q_{7890} &    0 
\end{bmatrix}.
\end{equation}
Here the $q_{ijkl}$ are maximal minors of $X$, so they
are dual Pl\"ucker coordinates on ${\rm Gr}(4,10)$.
The expansion into the entries of $X$ has degree $20$ and it 
is the sum of $18\,268\,320 $ monomials.

The formula (\ref{eq:AITformula}) appears in  \cite[Theorem 3.4.7]{AIT},
and it is derived as follows. Write the dual and primal
Pl\"ucker coordinates of the five lines
as the columns of the two $6 \times 5$ matrices
$$ \begin{matrix} \tilde X & = & \big[\, x_1 \wedge x_2 \qquad x_3 \wedge x_4 \qquad
x_5  \wedge  x_6  \qquad
x_7 \wedge x_8 \qquad
x_9 \wedge x_0 \,\bigr] \\
 (\tilde X)^* & = & \big[\, (x_1 \wedge x_2)^* \,\,\,\,\,(x_3 \wedge x_4)^* \,\,\,\,\,
(x_5  \wedge  x_6)^*  \,\,\,\,\,
(x_7 \wedge x_8)^* \,\,\,\,\,
(x_9 \wedge x_0)^* \,\bigr]. 
\end{matrix}
$$
The symmetric $5 \times 5$ matrix $ (\tilde X)^T \cdot (\tilde X)^*$
is precisely the matrix in (\ref{eq:AITformula}).
The left kernel of $\tilde X$ is spanned by
the vector $(L_{12},\ldots,L_{34})$ of
primal Pl\"ucker coordinates of the transversal line $L$.
Using Cramer's rule, we write each
$L_{ij}$ as a $5 \times 5$ minor of the matrix~$\tilde X$.
We compute
\begin{equation} \label{eq:CGr24} {\rm CL}_{\mathcal{V}_\beta} \,\, = \,\, 
{\rm C}_{{\rm Gr}(2,4)} \,\, = \,\,
L_{12} L_{34} \,-\,
L_{13} L_{24} \,+ \, L_{14} L_{23} \,\, = \,\, \frac{1}{2}\, {\rm det} \bigl(
 (\tilde X)^T \cdot (\tilde X)^* \bigr). 
 \end{equation}

 It remains to be seen that this is the Chow-Lam form of $\mathcal{V}_\beta
 = V(q_{12} , q_{34} ,q_{56} ,q_{78} , q_{90})$.
The matrix  $X$ represents a point in $\mathcal{CL}_{\mathcal{V}_\beta}$
if and only if its row space has a subspace $Q$ that lies in
 $ \mathcal{V}_\beta$.
This means that we can find a $2 \times 4$ matrix $T$ such that
the $2 \times 10$ matrix $Q = TX$ has its five minors $\{2i-1,2i\}$ vanish.
Then $L = \wedge_2 T$ is precisely the transversal line above. 
 \end{example}

\begin{proof}[Proof of Theorem \ref{thm:catalan}]
The matrix $X$ represents a point in $\mathcal{CL}_{\mathcal{V}_\beta}$
if and only if there exists a $2 \times s$ matrix $T$  of rank $2$ such that
the  matrix $TX$ has its consecutive minors $\{2i-1,2i\}$ vanish.
We view the entries of $T$ as dual Stiefel coordinates of a line in $\PP^{s-1}$,
 with dual Pl\"ucker coordinates $L = \wedge_2 T$.
 We view the columns   $x_{2i-1} \wedge x_{2i}$ of $\tilde X$ 
 as hyperplanes in the ambient space $\PP^{\binom{s}{2}-1}$
 of the Grassmannian ${\rm Gr}(2,s)$.
 Our condition says that these $t$ hyperplanes intersect ${\rm Gr}(2,s)$ in a common point $L$.
Since $t = 2s-3 = {\rm dim}({\rm Gr}(2,s))+1$, this imposes one constraint on $X$.
   The polynomial in the entries of $X$ that defines this hypersurface
is the Chow form of ${\rm Gr}(2,s)$, now evaluated at the
$t$ hyperplanes $x_{2i-1} \wedge x_{2i}$.
 \end{proof}

\begin{example}[Chow form of ${\rm Gr}(2,5)$]
We present a formula for ${\rm C}_{{\rm Gr}(2,5)}$ 
 in dual Stiefel coordinates, given by
three skew-symmetric $5 \times 5$ matrices $U, V, W$. 
We ask whether 
some linear combination $P = x U + y V + z W$ has rank $2$.
The five Pl\"ucker relations $ f_i = p_{jk} p_{lm} - p_{jl} p_{km} + p_{jm} p_{kl} $
are quadrics in $x,y,z$.  Let $f_6$ be the derivative of the
Jacobian determinant of $f_1,f_2,f_3$ with respect to $z$.
Following \cite[eqn (4.5)]{CBMS}, we
form the $6 \times 6$ matrix 
of coefficients from the ternary quadrics $f_1,f_2,\ldots,f_6$.
Its determinant has degrees $(5,5,6)$ in $ (U,V,W)$,
and it equals  ${\rm C}_{{\rm Gr}(2,5)}$ times
$w_{45} $. Note that ${\rm C}_{{\rm Gr}(2,5)}$ has degree $5$
in Pl\"ucker coordinates.

By rewriting this Chow form in primal Pl\"ucker coordinates,
as above, we obtain a formula for the Chow-Lam form
${\rm CL}_{\mathcal{V}_\beta} $, where
$\beta = (2,2,2,2,2,2,2)$.
This has degree $14$ in Pl\"ucker coordinates on
${\rm Gr}(5,14)$. At present we do not know
any determinantal formula like (\ref{eq:AITformula}).

It remains a challenge to find practical tools for computing
the Chow forms of ${\rm Gr}(2,n)$ when $n \geq 6$.
A formula for the power $({\rm C}_{{\rm Gr}(2,n)})^{n-2}$ in dual Stiefel coordinates
appears  in \cite[Section 3]{AFRW}.
This is based on  resonance varieties for Koszul modules.
It would be desirable to implement and test this formula.
Can the exponent $n-2$ be removed in this construction?
\end{example}

\begin{remark} \label{rmk:Mnkt}
Theorem \ref{thm:catalan} generalizes to a
positroid $M$ on $n=kt$ elements in rank $k \geq 3$.
It has precisely $t$ nonbases which are pairwise disjoint.
We assume that $t = (s-k)k+1$ for some integer $ s > k$.
The Chow-Lam degree of the positroid variety $\mathcal{V}_M \subset {\rm Gr}(k,n)$ equals
$$ \lambda(\mathcal{V}_M) \,\,=\,\, \frac{n}{s} \cdot {\rm degree} \, {\rm Gr}(k,s) , $$
and we can derive ${\rm CL}_{\mathcal{V}_M}$
from ${\rm C}_{{\rm Gr}(k,s)}$ as above. The coordinates are the entries of
an $s \times n$ matrix $X$ and an $\binom{s}{k} \times t$
matrix $\widetilde X$, obtained by wedging the $k$-clusters of columns in $X$.

For example, if $k=3,s=6$ then $t = 10$, $n=30$, and
$\mathcal{V}_M$ is a variety of dimension  $71$ in ${\rm Gr}(3,30)$.
Its Chow-Lam degree equals $\,210 = \frac{n}{s} \cdot 42 = 
\frac{n}{s} \cdot {\rm degree} \,{\rm Gr}(3,6)$. The
$24 \times 30$ matrix $Z$ maps $\mathcal{V}_M$ to a hypersurface
in ${\rm Gr}(3,24)$, with equation
found from ${\rm CL}_{\mathcal{V}_M}$ by Corollary~\ref{cor:projhyper}.
\end{remark}

In this section we presented techniques for studying the
Chow-Lam forms of matroid varieties $\mathcal{V}_M$.
A matrix $X$ of dual Stiefel coordinates represents
a point in $\mathcal{CL}_{\mathcal{V}_M}$ if 
there exists a matrix $T$ such that $T X $ realizes
the matroid $M$. We can compute
${\rm CL}_{\mathcal{V}_M}$ by eliminating
$T$ from the equations $TX \in \mathcal{V}_M$.
It can be preferable to perform the elimination
directly with  dual Pl\"ucker coordinates, i.e.~with
the maximal minors of $X$. 
This approach is recommended
when $\mathcal{V}_M$ is given parametrically.
We conclude the section with such an example.

\begin{example}[A rank $3$ matroid]
Fix the matroid $M = \{126, 135, 234, 456\}$. Here $k=3$, $n=6$, and
$r = 5$, so $\mathcal{V}_M $ has codimension $4$ in ${\rm Gr}(3,6)$,
and it projects to a hypersurface in 
${\rm Gr}(3,5)$. 
Its Chow-Lam locus $\mathcal{CL}_{\mathcal{V}_M}$ is
a subvariety of ${\rm Gr}(4,6)$, given by the parametrization
$$ X \,\,\, = \,\,\, \begin{small} \begin{bmatrix}
   \,  x_1 & 0 & 0 & 0 & \!\! -x_5 & x_6 \\
   \,  0 & x_2 & 0 & x_4 & 0 & \!\! -x_6 \, \\
   \,  0 & 0 & x_3 & \! \! -x_4 & x_5 & 0 \\
   \,  y_1 & y_2 & y_3 & y_4 & y_5 & y_6
\end{bmatrix}. \end{small}
$$
Indeed,  the first three rows parametrize $\mathcal{V}_M$.
The maximal minors satisfy the cubic relation
$$ {\rm CL}_{\mathcal{V}_M} \,\, = \,\,
 q_{1234} q_{1356} q_{2456}
+ q_{1235} q_{1246} q_{3456}
 - q_{1235} q_{1346} q_{2456}. $$
This form is unique modulo the Pl\"ucker relations,
and is easily found by elimination.
\end{example}

\section{Hurwitz-Lam Forms and Beyond}
\label{sec5}

The purpose of this section is to generalize the 
Hurwitz form \cite{Hur} for subvarieties of projective space to 
subvarieties of Grassmannians.
Theorem \ref{thm:hurwitzprojection}
extends to that setting and
provides a tool for computing the
branch loci of projections
of subvarieties of Grassmannians.
At the end of this section, we generalize even further
by introducing higher Chow-Lam forms.

Let $\mathcal{V}$ be a subvariety of ${\rm Gr}(k,n)$
of dimension $k(r-k)$ for some $r  \in \{k+1,\ldots,n\}$.
Fix a general point $P \in {\rm Gr}(k+n-r,n)$. This
represents a linear subspace in $\CC^n$, and we consider the
Grassmannian ${\rm Gr}(k,P)$, which parametrizes  all linear
subspaces $Q \in {\rm Gr}(k,n)$ such that $Q \subseteq P$.
We are interested in the intersection of $\mathcal{V}$ with 
the subGrassmannian ${\rm Gr}(k,P)$.

\begin{lemma} The intersection 
\begin{equation} \label{eq:finiteintersection} \mathcal{V} \, \, \cap \,\, {\rm Gr}(k,P)  \end{equation}
is a finite set of points. The number of points
is the coefficient $\delta_I(\mathcal{V})$ in
(\ref{eq:schubert}), where $I = \{\,r{-}k+1, \,r{-}k+2,\,\ldots,\,r\,\}$.
We call $\,\delta_I(\mathcal{V})$ the {\em Hurwitz-Lam degree} of the variety $\mathcal{V}$.
\end{lemma}

\begin{proof} The proof is a dimension argument,
like that leading to Lemma \ref{lem:CLproper}.
The identification  of the specific Schubert class
$\mathcal{S}_I$ is analogous to that in
the derivation of Theorem \ref{thm:CL}.
\end{proof}

We define the {\em Hurwitz-Lam locus} of the given variety $\mathcal{V}$ as follows:
\begin{equation}
\label{eq:HLV} \mathcal{HL}_\mathcal{V} \,\, = \,\,
\bigl\{ \, P \in {\rm Gr}(k+n-r,n) \,: \, \hbox{the intersection (\ref{eq:finiteintersection}) is not transverse} \bigr\} . 
\end{equation}
If the Hurwitz-Lam degree is at least two then 
$\mathcal{HL}_\mathcal{V}$ is a hypersurface in  ${\rm Gr}(k+n-r,n)$.
The polynomial  in Pl\"ucker coordinates that defines this hypersurface is the
 {\em Hurwitz-Lam form}, denoted ${\rm HL}_{\mathcal{V}}$.
 If the Hurwitz-Lam locus is not a hypersurface then we set
 ${\rm HL}_{\mathcal{V}} = 1$.

The Hurwitz-Lam form ${\rm HL}_{\mathcal V}$
of a subvariety $\mathcal{V} \subset {\rm Gr}(k,n)$ can be
used to describe the branch loci of projections to
smaller Grassmannians. The set-up is as in 
Proposition  \ref{prop:chowprojection}.

\begin{theorem}[Computing Branch Loci] \label{thm:branch}
Let $\mathcal{V}$ be a variety in ${\rm Gr}(k,n)$ whose
Hurwitz-Lam degree $\delta_I(\mathcal{V})$ is at least two. 
A general linear projection $Z $ maps $\mathcal{V}$
onto ${\rm Gr}(k,r)$. The branch locus is a hypersurface
in ${\rm Gr}(k,r)$, and its equation is obtained
from  the Hurwitz-Lam form ${\rm HL}_\mathcal{V}$
by replacing the primal Pl\"ucker coordinates with twistor 
coordinates as in (\ref{eq:twistor2}).
\end{theorem}

\begin{proof}
A point $Y \in {\rm Gr}(k,r)$ lies in the branch locus of the map $Z$ from $\mathcal{V}$
if and only if the fiber $Z^{-1}(Y)$ intersects
$\mathcal{V}$ non-transversally. This happens
if and only if $Z^{-1}(Y)$ is in $\mathcal{HL}_\mathcal{V}$.
That the branch locus is codimension one
follows from the Zariski-Nagata Purity Theorem.
This  ensures that $\mathcal{HL}_\mathcal{V}$ has codimension one.
We obtain the equation of the branch locus from ${\rm HL}_\mathcal{V}$
by the same arguments as in Proposition \ref{prop:chowlamprojection}.
For $k=1$ see Theorem \ref{thm:hurwitzprojection}.
\end{proof}

We illustrate Theorem \ref{thm:branch} with a case study
that is motivated by  amplituhedra and other
Grassmann polytopes \cite{LamCurrent}.
We set $k=2, n=8$ and we let
$\mathcal{V} = \mathcal{V}_\beta = V(q_{12},q_{34},q_{56},q_{78})$ be the positroid variety 
in ${\rm Gr}(2,8)$ given by the partition $\beta = (2,2,2,2)$.
We have ${\rm codim}(\mathcal{V})=4$
and ${\rm dim}(\mathcal{V}) = 8$. 
The Hurwitz-Lam degree of $\,\mathcal{V}$ is two. 
This means that a  general linear projection 
$Z:{\rm Gr}(2,8) \dashrightarrow {\rm Gr}(2,6)$
induces a $2$-to-$1$ map from $\mathcal{V}$ 
onto ${\rm Gr}(2,6)$.
The branch locus of this map is a hypersurface in ${\rm Gr}(2,6)$.
The next proposition gives  its~equation.

\begin{proposition} \label{prop:G28}
	The Hurwitz-Lam form of the positroid variety $\mathcal{V} \subset {\rm Gr}(2,8)$ equals
	\begin{equation}
		\label{eq:G28formula} 
			{\rm HL}_\mathcal{V} \,\,\,= \,\,\,
		 {\rm det} \,\begin{bmatrix} 
			\,0  & q_{1234} & q_{1256} & q_{1278} \,   \\
			\,q_{1234} &    0  & q_{3456} & q_{3478} \, \\
\,			q_{1256} & q_{3456} &    0  & q_{5678} \,  \\
\,			q_{1278} & q_{3478} & q_{5678} &    0  \,
		\end{bmatrix}.
	\end{equation}
	This formula is invariant under the duality map on ${\rm Gr}(4,8)$, so it works
	in both primal and dual Pl\"ucker coodinates. The branch locus of $Z$
	in ${\rm Gr}(2,6)$ is a quartic hypersurface, whose equation is obtained by
	substituting the $70$ twistor coordinates for the $p_{ijkl}$ in ${\rm HL}_\mathcal{V}$.
\end{proposition}

\begin{example} \label{ex:dmitrii}
	Consider the map ${\rm Gr}(2,8) \rightarrow {\rm Gr}(2,6)$ given by the totally positive matrix
	$$ 
	Z \quad = \quad \begin{bmatrix} 
		\,      1 & 0 & 0 & 0 & 0 & 0 & -1 & -6 \, \, \\
		\,             0 & 1 & 0 & 0 & 0 & 0 &  \phantom{-} 1 & \phantom{-} 5 \, \, \\
		\,         0 & 0 & 1 & 0 & 0 & 0 & -1 & -4 \, \, \\
		\,             0 & 0 & 0 & 1 & 0 & 0 & \phantom{-} 1 & \phantom{-} 3 \, \, \\
		\,             0 & 0 & 0 & 0 & 1 & 0 & -1 & -2 \, \, \\
		\,             0 & 0 & 0 & 0 & 0 & 1 &  \phantom{-}1 &  \phantom{-}1 
	\end{bmatrix}.
	$$             
	The image of the positroid cell $\mathcal{V}_{\geq 0}$ under $Z$
	is an interesting Grassmann polytope. The quartic branch locus described above
	contributes a piece to the boundary of this object.
	Amplituhedron experts are interested in the quartic equation defining that branch locus.
	
	Let $Y$ be a $6 \times 2$ matrix of unknowns,
	and write $y_{ij}$ for the $2 \times 2$ minors of $Y$, where $1\! \leq \! i \! < \! j \! \leq\! 6$. The substitution (\ref{eq:twistor2}) from primal Pl\"ucker coordinates to twistor coordinates~is
	$$ \begin{matrix}
		p_{1234} = y_{56}, \,p_{1235} = -y_{46}, \,p_{1236} = y_{45}, \,p_{1237} = y_{45}+y_{46}+y_{56},\,
		p_{1238} = y_{45}+2 y_{46}+3 y_{56}, \,\ldots,\\
		p_{4678} = 2y_{12}+3 y_{13}-y_{15}+4 y_{23}-2 y_{25}-y_{35}\, , \,\,
		p_{5678} = y_{12}+2 y_{13}+y_{14}+3 y_{23}+2 y_{24}+y_{34}.
	\end{matrix}
	$$
	After this substitution, ${\rm HL}_\mathcal{V}$ is a quartic 
	in dual Pl\"ucker coordinates $y_{12}, y_{13}, \ldots, y_{56}$ on ${\rm Gr}(2,6)$.
	Its unique expansion in terms of standard monomials has $126$ terms. It looks like
	$$ \begin{matrix}
		10 y_{12}^2 y_{34} y_{56}-4 y_{12}^2 y_{35}^2-4 y_{12}^2 y_{35} y_{36}
		-12 y_{12}^2 y_{35} y_{45}-14y_{12}^2 y_{35} y_{46}
		-y_{12}^2 y_{36}^2-4 y_{12}^2 y_{36} y_{46}
		-9 y_{12}^2 y_{45}^2 \\
		-\,12 y_{12}^2 y_{45} y_{46}
		-4 y_{12}^2 y_{46}^2
		+16 y_{12} y_{13} y_{34} y_{56}
		-16 y_{12} y_{13} y_{35} y_{45}
		-20 y_{12} y_{13} y_{35} y_{46} \,+\,\, \cdots \, \cdots .\end{matrix} $$
	Recall that a monomial $y_{i_1 j_1} y_{i_2 j_2} y_{i_3 j_3} y_{i_4 j_4}$ is {\em standard} if 
	$i_1 \! \leq \! i_2 \!\leq \! i_3 \! \leq \! i_4$ and $j_1 \! \leq \! j_2 \! \leq \! j_3 \! \leq \! j_4$.
	This quartic equation with $126$ standard monomials defines the branch locus in ${\rm Gr}(2,6)$.
\end{example}

\begin{proof}[Proof of Proposition \ref{prop:G28}]
	One checks that the expression on the right hand side of (\ref{eq:G28formula})
	is invariant under the duality map given by $\,p_{\sigma_1 \sigma_2 \sigma_3 \sigma_4}
	\mapsto {\rm sign}(\sigma)\, p_{\sigma_5 \sigma_6 \sigma_7 \sigma_8}$
	for any permutation $\sigma $ of $\{1,2,\ldots,8\}$.
	We may therefore prove (\ref{eq:G28formula}) for the dual Pl\"ucker coordinates on ${\rm Gr}(4,8)$.
	
	Any element $P$ of ${\rm Gr}(4,8)$ can be represented as the row span of a rank $4$ matrix
	$$
	X \,\,\, = \,\,\,
	\begin{bmatrix} \,x_{11} & x_{12} & x_{13} & x_{14} & x_{15} & x_{16} & x_{17} & x_{18} \, \\
		\,   x_{21} & x_{22} & x_{23} & x_{24} & x_{25} & x_{26} & x_{27} & x_{28} \, \\
		\, x_{31} & x_{32} & x_{33} & x_{34} & x_{35} & x_{36} & x_{37} & x_{38} \, \\
		\,  x_{41} & x_{42} & x_{43} & x_{44} & x_{45} & x_{46} & x_{47} & x_{48}\,
	\end{bmatrix}.
	$$ 
	The Pl\"ucker coordinates $p_{ijkl}$ of $P$ are the $70$ maximal minors of $X$.
	For a general subspace $P$ of dimension $4$ in $\CC^8$, consider all
	$2$-dimensional subspaces $Q \subset P $ such that
	$Q$ is in~$\mathcal{V}$. We claim that there are precisely
	two such subspaces $Q$.  In symbols,
	$\# \bigl( \mathcal{V}  \,\cap \, {\rm Gr}(2,P ) \bigr) = 2$. We are interested in $P$ such that this number is one.
	
	The argument is analogous to Example \ref{ex:transversal},
	but now for transversal lines to four lines. Let $x_j$ denote the $j$th column of $X$, viewed as a point in $\PP^3$. We consider $2 \times 4$ matrices
	$T$ such that the $2 \times 8$  matrix $TX$ satisfies
	$q_{12} = q_{34} = q_{56} = q_{78} = 0$. 
	The rows of $T$ span a line in $\PP^3$ which intersects the four lines 
	$ \overline{ x_1  x_2}$,
	$\overline{ x_3 x_4}$,
	$ \overline{x_5 x_6}$ and
	$ \overline{x_7 x_8}$ in $\PP^3$.
	Then $X$ is in the Hurwitz-Lam locus whenever only one such line $T$ exists. Now consider the $6 \times 4$ matrices
	$$ \begin{matrix} \tilde X & = & \big[\, x_1 \wedge x_2 \qquad x_3 \wedge x_4 \qquad
		x_5  \wedge  x_6  \qquad
		x_7 \wedge x_8 \,\bigr] \\
		(\tilde X)^* & = & \big[\, (x_1 \wedge x_2)^* \,\,\,\,\,(x_3 \wedge x_4)^* \,\,\,\,\,
		(x_5  \wedge  x_6)^*  \,\,\,\,\,
		(x_7 \wedge x_8)^* \,\bigr], 
	\end{matrix}
	$$
	 where we order the six rows of each matrix using the ordering  $[a_{12}:a_{13}:a_{23}:a_{14}:a_{24}:a_{23}]$ of Pl\"ucker coordinates of $x_i \wedge x_{i+1}$. The left kernel of $\tilde{X}$ is represented by a $2 \times 6$ matrix, and gives a line $L$ in $\PP^5.$ The lines $\overline{ x_1  x_2}, \ldots , \overline{x_7 x_8}$ have a single common transversal if and only if $L$ is tangent to ${\rm Gr}(2,4)$ in its Pl\"ucker embedding; in other words, if $L$ is in the Hurwitz locus of ${\rm Gr}(2,4).$ We compute the Hurwitz form of ${\rm Gr}(2,4)$ to be
	 \begin{equation}\label{eq:jac}
	 	\sum_{1 \leq i < j \leq 6} (-1)^{\sigma(i,j)} L_{i,j}L_{7-j,7-i} ,
	 \end{equation}
	 where $L_{i,j}$ are the dual Pl\"ucker coordinates of $L$ and $\sigma(i,j)$ is the number of occurences of $2$ or $5$ in $(i,j).$ Note that $L_{i,j}$ are actually the primal
	 Pl\"ucker coordinates $p_{i,j}(\tilde{X}^T)$ for the
	 row span of the matrix $\tilde{X}^T$ in ${\rm Gr}(4,6).$ By the Cauchy-Binet formula, we see
	 $$\det( (\tilde{X})^T  \cdot (\tilde{X})^*) \ =  \ \sum_{1 \leq i< j\leq 6} p_{i,j}(\tilde{X}^T)p_{i,j}((\tilde{X}^*)^T),$$
	 which is equal to the expression in (\ref{eq:jac}).  Finally, by Laplace expansion, the $4 \times 4$ matrix $(\tilde{X})^T \cdot (\tilde{X})^*$ is precisely the matrix on the right side of (\ref{eq:G28formula}). 
\end{proof}

\begin{remark}
Proposition \ref{prop:G28} extends to all positroid varieties
$\mathcal{V}_\beta$ where $\beta = (2,2,\ldots,2)$ is a partition of
$n=2t$ with $t = 2s-4$ even. The Hurwitz-Lam form
${\rm H}_{\mathcal{V}_\beta}$ is obtained from the Hurwitz form of ${\rm Gr}(2,s)$
by the  same substitution as that for Chow forms in Theorem~\ref{thm:catalan}.
Extending to the setting of
Remark \ref{rmk:Mnkt}, we can consider the Hurwitz form
of any Grassmannian ${\rm Gr}(k,s)$. This becomes the
Hurwitz-Lam form for a postroid whose
ground set is the disjoint union of nonbases.
\end{remark}

We now shift gears and briefly return to the setting
of Section \ref{sec:two}. For any subvariety $\mathcal{V}$ of $\PP^{n-1}$ of
dimension $d$
 and any integer $p \leq d+1$,
one associates a subvariety in a Grassmannian:
\begin{equation}
\label{eq:CpV}
 \mathcal{C}_p(\mathcal{V}) 
\,\, = \,\, \bigl\{ \, L \in {\rm Gr}(n-p,n) \,:\,
\hbox{the intersection $\, \mathcal{V}\,\cap \, L\,$ is not transverse} \,\bigr\}.
\end{equation}
For $p=d+1$ this is the Chow locus,
for $p=d$ it is the Hurwitz locus, and for $p=1$ it is the discriminant.
These varieties $\mathcal{C}_p(\mathcal{V})$ are usually hypersurfaces, and they are known~as {\em coisotropic hypersurfaces} or {\em higher Chow forms} of $\mathcal{V}$.
They were studied by Gel'fand, Kapranov and Zelevinsky 
in \cite[Chapter 4]{GKZ} and the theory was further developed by
Kohn in~\cite{kohn}. 

It is natural to generalize the definition of higher Chow forms 
to the setting of Grassmannians. Let $\mathcal{V}$ be a subvariety of
${\rm Gr}(k,n)$. For any sufficiently small integer $p$, we consider
\begin{equation}
\label{eq:CLpV}
 \mathcal{CL}_\mathcal{V}^{(p)} 
\,\, = \,\, \bigl\{ \, L \in {\rm Gr}(n-p,n) \,:\,
\hbox{the intersection $\, \mathcal{V} \cap {\rm Gr}(k,L) \,$ is not transverse} \,\bigr\}.
\end{equation}
These varieties are called {\em higher Chow-Lam loci}.
Of particular interest are the cases when $\mathcal{CL}_\mathcal{V}^{(p)}$
is a hypersurface, defined by a single equation ${\rm CL}_\mathcal{V}^{(p)}$.
We call these the {\em higher Chow-Lam forms} of $\mathcal{V}$.
The case when $p = r-k$ and ${\rm dim}(\mathcal{V}) = kp+1$
recovers the Chow-Lam form; see (\ref{eq:CLdef2}).
The Hurwitz-Lam form arises for $p = r-k$ and ${\rm dim}(\mathcal{V}) = kp$, 
see~(\ref{eq:HLV}).

\begin{example}[$p=2$]
Let $\mathcal{V}_\beta  = V(q_{12},q_{34},q_{56}) \subset {\rm Gr}(2,6)$
be the codimension $3$ positroid variety given by $\beta = (2,2,2)$.
Its Chow-Lam form was derived in Example \ref{ex:positroid265}.
We now compute the higher Chow-Lam form
${\rm CL}_{\mathcal{V}_\beta}^{(2)}$.
The locus $\mathcal{CL}_{\mathcal{V}_\beta}^{(2)} $ consists of all
subspaces $L \in {\rm Gr}(4,6)$ such that the curve
$\mathcal{V}_\beta \cap {\rm Gr}(2,L)$ is singular.
This corresponds to cutting the Grassmannian
${\rm Gr}(2,4)$  with three hyperplanes in $\PP^5$
such that the intersection is a singular curve.
The corresponding equation ${\rm C}_{{\rm Gr}(2,4)}^{(3)}  $ 
  is irreducible of degree $6$  in primal Stiefel coordinates.
If we replace these by the columns of the $6 \times 3$ matrix
  $\tilde X = \bigl[ x_1 \!\wedge\! x_2 \,\,\, x_3 \!\wedge \! x_4 \,\,\, x_5 \!\wedge \! x_6 \bigr]$,
  then we obtain a reducible polynomial of degree $12$
  in the entries of the $ 4 \times 6$ matrix $X = [x_1,x_2,\ldots,x_6]$:
  $$ {\rm CL}_{\mathcal{V}_\beta}^{(2)}\,\, = \,\,
 {\rm C}_{{\rm Gr}(2,4)}^{(3)}                             \,\, = \,\,
 p_{12}\,p_{34}\,p_{56} \,\, = \,\,
 {\rm det}(x_3,x_4,x_5,x_6) \,
 {\rm det}(x_1,x_2,x_5,x_6) \,
 {\rm det}(x_1,x_2,x_3,x_4) .
$$
This is similar to the identity  in (\ref{eq:CGr24}).
   The factorization has the
   following geometric meaning: 
   Consider the curve of lines that 
intersect three given lines in $\PP^3$.
This curve is smooth if and only if the three lines are
pairwise disjoint.
Thus $\,\mathcal{CL}_{\mathcal{V}_\beta}^{(2)}\, $ has three irreducible components.
\end{example}

                \bigskip
		\bigskip

		\footnotesize
                \noindent {\bf Authors' addresses:}

                \smallskip

                \noindent Elizabeth Pratt, UC Berkeley
                \hfill \url{epratt@berkeley.edu}

                \noindent  Bernd Sturmfels, MPI-MiS Leipzig
                \hfill \url{bernd@mis.mpg.de}

\end{document}